\documentclass[11pt]{article}
\usepackage{amsmath,amssymb,amsthm}
\newcommand{\ddd}{
\text{\begin{picture}(12,8)
\put(-2,-4){$\cdot$}
\put(3,0){$\cdot$}
\put(8,4){$\cdot$}
\end{picture}}}

\renewcommand{\le}{\leqslant}
\renewcommand{\ge}{\geqslant}
\newtheorem{theorem}{Theorem}

\newtheorem{example}[theorem]{Example}

\newtheorem{lemma}[theorem]{Lemma}

\begin{document}

\title{Canonical forms for complex matrix congruence and *congruence\footnotetext{This is the authors' version of a work that was published in Linear Algebra Appl. 416 (2006) 1010--1032.}}
\author{Roger A. Horn\\Department of Mathematics, University of Utah\\Salt Lake City, Utah 84103, rhorn@math.utah.edu
\and Vladimir V. Sergeichuk\thanks{The research was started while this author was
visiting the University of Utah supported by NSF grant DMS-0070503.}\\Institute of Mathematics, Tereshchenkivska 3\\Kiev, Ukraine, sergeich@imath.kiev.ua}
\date{}
\maketitle

\begin{abstract}
Canonical forms for
congruence and *congruence of
square complex matrices were
given by Horn and Sergeichuk
in [Linear Algebra Appl. 389
(2004) 347--353], based on
Sergeichuk's paper [Math.
USSR, Izvestiya 31 (3) (1988)
481--501], which employed the
theory of representations of
quivers with involution. We
use standard methods of
matrix analysis to prove
directly that these forms are
canonical. Our proof provides
explicit algorithms to
compute all the blocks and
parameters in the canonical
forms. We use these forms to
derive canonical pairs for
simultaneous congruence of
pairs of complex symmetric
and skew-symmetric matrices
as well as canonical forms
for simultaneous *congruence
of pairs of complex Hermitian
matrices.

\textit{AMS classification:} 15A21; 15A63

\textit{Keywords:} Canonical forms, Congruence, *Congruence, Bilinear forms;
Sesquilinear forms, Canonical pairs.

\end{abstract}

\renewcommand{\maltese}{\text{\setlength{\fboxsep}{1pt}\fbox{$*$}}}
\renewcommand{\le}{\leqslant} \renewcommand{\ge}{\geqslant}

\section{Introduction}

\label{s1}

Canonical matrices for congruence and *congruence over any field
$\mathbb{F}$ of characteristic not 2 were established in
\cite[Theorem 3]{ser1} up to classification of Hermitian forms
over finite extensions of $\mathbb{F}$. Canonical forms for
complex matrix congruence and *congruence are special cases of the
canonical matrices in \cite{ser1} since a classification of
Hermitian forms over the complex field is known. Simpler versions
of these canonical forms were given in \cite{hor-ser}, which
relied on \cite{ser1}, and hence on the theory of representations
of quivers with involution on which the latter is based.

In this paper, all matrices considered are complex. We use standard tools of
matrix analysis to give a direct proof that the complex matrices given in
\cite{hor-ser} are canonical for congruence and *congruence.

Let $A$ and $B$ be square complex matrices of the same size. We
say that $A$ and $B$ are \emph{congruent} if there is a
nonsingular $S$ such that $S^{T}AS=B$; they are *\negthinspace
\emph{congruent} if there is a nonsingular $S$ such that
$S^{\ast}AS=B$. We let $S^{\ast}:=[\bar{s}_{ji}]=\bar{S}^{T}$
denote the complex conjugate transpose of $S=[s_{ij}]$ and write
$S^{-T}:=(S^{-1})^{T}$ and $S^{-\ast}:=(S^{-1})^{\ast}$.

Define the $n$-by-$n$ matrices
\begin{equation}
\Gamma_{n}=%
\begin{bmatrix}
0 &  &  &  &  & (-1)^{n+1}\\
&  &  &  & \ddd
& (-1)^{n}\\
&  &  & -1 &
\ddd & \\
&  & 1 & 1 &  & \\
& -1 & -1 &  &  & \\
1 & 1 &  &  &  & 0
\end{bmatrix}
\quad\text{(}\Gamma_{1}=[1]\text{),} \label{Gamman}%
\end{equation}%
\begin{equation}
\Delta_{n}=%
\begin{bmatrix}
0 &  &  & 1\\
&  &
\ddd & i\\
& 1 &
\ddd & \\
1 & i &  & 0
\end{bmatrix}
\quad\text{(}\Delta_{1}=[1]\text{),} \label{Deltan}%
\end{equation}
and the $n$-by-$n$ Jordan block with eigenvalue $\lambda$%
\[
J_{n}(\lambda)=%
\begin{bmatrix}
\lambda & 1 &  & 0\\
& \lambda & \ddots & \\
&  & \ddots & 1\\
0 &  &  & \lambda
\end{bmatrix}
\quad\text{(}J_{1}(\lambda)=[\lambda]\text{).}%
\]
The most important properties of these matrices for our purposes are that
$\Gamma_{n}$ is real and $\Gamma_{n}^{-T}\Gamma_{n}=\Gamma_{n}^{-\ast}%
\Gamma_{n}$ is similar to $J_{n}((-1)^{n+1})$; $\Delta_{n}$ is symmetric and
$\Delta_{n}^{-\ast}\Delta_{n}=\bar{\Delta}_{n}^{-1}\Delta_{n}$ is similar to
$J_{n}(1)$.

We also define the $2n$-by-$2n$ matrix%
\begin{equation}
H_{2n}(\mu)=%
\begin{bmatrix}
0 & I_{n}\\
J_{n}(\mu) & 0
\end{bmatrix}
\quad\bigr(H_{2}(\mu)=\left[
\begin{array}
[c]{cc}%
0 & 1\\
\mu & 0
\end{array}
\right]  \bigl)\text{,} \label{Hn}%
\end{equation}
the \emph{skew sum} of $J_{n}(\mu)$ and $I_{n}$. If $\mu\neq0$, then
$H_{2n}(\mu)^{-T}H_{2n}(\mu)$ is similar to $J_{n}(\mu)\oplus J_{n}(\mu^{-1})$
and $H_{2n}(\mu)^{-\ast}H_{2n}(\mu)$ is similar to $J_{n}(\mu)\oplus
J_{n}(\bar{\mu}^{-1})$.

Sylvester's Inertia Theorem describes the *congruence canonical
form of a complex Hermitian matrix. Our main goal is to give a
direct proof of the following theorem, which generalizes
Sylvester's theorem to all square complex matrices.

\begin{theorem}
[{\cite[Section 2]{hor-ser}}]
\label{t2}
\textrm{(a)} Each square complex
matrix is congruent to a direct sum, uniquely determined up to permutation of
summands, of canonical matrices of the three types%
\begin{equation}\label{table1}
\renewcommand{\arraystretch}{1.2}
\begin{tabular}
[c]{|c|c|}\hline $\text{Type
0}$ & $J_{n}(0)$\\\hline
$\text{Type I}$ &
$\Gamma_{n}$\\\hline
$\text{Type II}$ &
$H_{2n}(\mu),\ \
0\neq\mu\neq(-1)^{n+1},$
  \\
&$\mu$ is determined up to
replacement by $\mu^{-1}$
\\\hline
\end{tabular}
\end{equation}

\textrm{(b)} Each square complex matrix is \textrm{*\negthinspace}congruent to
a direct sum, uniquely determined up to permutation of summands, of canonical
matrices of the three types
\begin{equation}\label{table2}
\renewcommand{\arraystretch}{1.2}
\begin{tabular}
[c]{|c|c|}\hline $\text{Type
0}$ & $J_{n}(0)$\\\hline
$\text{Type I}$ &
$\lambda\Delta_{n},\ \
|\lambda|=1$\\\hline
$\text{Type II}$ &
$H_{2n}(\mu),\ \
|\mu|>1$\\\hline
\end{tabular}
\end{equation}
Instead of $\Delta_{n}$, one
may use $\Gamma_{n}$ or any
other nonsingular $n\times n$
matrix $F_{n}$ for which
there exists a real
$\theta_{n}$ such that
$F_{n}^{-\ast}F_{n}$ is
similar to
$J_{n}(e^{i\theta_{n}})$.
\end{theorem}

For *congruence canonical matrices of Type I, it is sometimes
convenient to identify the unit-modulus canonical parameter
$\lambda$ with the \emph{ray} $\{t\lambda:t>0\}$ or with the
\emph{angle }$\theta$ such that $\lambda =e^{i\theta}$ and
$0\leq\theta<2\pi$. If $\lambda$ occurs as a coefficient of exactly
$k$ blocks $\Delta_{n}$ in a
*congruence canonical form,
we say that it is a
\emph{canonical angle (or
ray) of order }$n$ \emph{with
multiplicity }$k$.

Our proof of Theorem \ref{t2} provides explicit algorithms to compute the
sizes and multiplicities of the canonical blocks $J_{n}(0)$, $\Gamma_{n}$,
$\lambda\Delta_{n}$, and $H_{2n}(\mu)$ and their canonical parameters
$\lambda$ and $\mu$.

It suffices to prove Theorem \ref{t2} only for nonsingular matrices because of
the following lemma, which is a specialization to the complex field of a
\textit{regularizing decomposition} for square matrices over any field or skew
field with an involution \cite{hor-ser1}.

\begin{lemma}
\label{t3} Each square
complex matrix $A$ is
congruent (respectively,
{\rm*\!}congruent) to a
direct sum of the form
\begin{equation}
B\oplus J_{r_{1}}(0)\oplus\cdots\oplus J_{r_{p}}(0)\quad\text{with a
nonsingular $B$.} \label{1.1}%
\end{equation}
This direct sum is uniquely
determined up to permutation
of its singular direct
summands and replacement of
$B$ by any matrix that is
congruent (respectively,
{\rm*\!}congruent) to it.
\end{lemma}

The nonsingular direct summand $B$ in (\ref{1.1}) is called the \emph{regular
part} of $A$; the singular blocks in (\ref{1.1}) are its Type 0 blocks. There
is a simple algorithm to determine all of the direct summands in (\ref{1.1}).
If desired, this algorithm can be carried out using only unitary
transformations \cite{hor-ser1}.

In our development, it is convenient to use some basic properties of
\textit{primary matrix functions}. For a given square complex matrix $A$ and a
given complex valued function $f$ that is analytic on a suitable open set
containing the spectrum of $A$, the primary matrix function $f(A)$ may be
defined using a power series, an explicit formula involving the Jordan
canonical form of $A$, or a contour integral. For our purposes, its most
important property is that for each $A$, $f(A)$ is a polynomial in $A$ (the
polynomial may depend on $A$, however), so $f(A)$ commutes with any matrix
that commutes with $A$. For a systematic exposition of the theory of primary
matrix functions, see \cite[Chapter 6]{horn}.

Preceded by Weierstrass, Kronecker developed a comprehensive theory of
equivalence of matrix pencils in the late nineteenth century, but a similarly
complete theory of matrix congruence has been achieved only recently. Gabriel
\cite{gab} reduced the problem of equivalence of bilinear forms to the problem
of equivalence of nonsingular bilinear forms. Riehm \cite{rie} reduced the
problem of equivalence of nonsingular bilinear forms to the problem of
equivalence of Hermitian forms. His reduction was improved and extended to
sesquilinear forms in \cite{RSF}.

Using Riehm's reduction, Corbas and Williams \cite{cor} studied
canonical forms for matrix congruence over an algebraically closed
field with characteristic not 2. However, their proposed
nonsingular canonical matrices are cumbersome and not canonical,
e.g., their matrices
\[
A=\left[
\begin{array}
[c]{cc}%
0 & 1\\
1/2 & 0
\end{array}
\right] \quad\text{and}\quad
B=\left[
\begin{array}
[c]{cc}%
0 & 1\\
2 & 0
\end{array}
\right]
\]
are actually congruent: $BAB^{T}=B$. For the singular case, they refer to the
list of \textquotedblleft singular blocks of known type\textquotedblright\ in
\cite[p. 60]{wat}. These singular blocks are canonical but cumbersome, and we
are fortunate that they may be replaced by the set of singular Jordan blocks;
see \cite{ser1} or \cite{hor-ser1}.

Any square complex matrix $A$ can be represented uniquely as $A=\mathcal{S}%
+\mathcal{C}$, in which $\mathcal{S}$ is symmetric and $\mathcal{C}$
is skew-symmetric; it can also be represented uniquely as
$A=\mathcal{H}+i\mathcal{K}$, in which both $\mathcal{H}$ and
$\mathcal{K}$ are Hermitian. A simultaneous congruence of
$\mathcal{S}$ and $\mathcal{C}$ corresponds to a congruence of $A$
and a simultaneous
*congruence of $\mathcal{H}$ and $\mathcal{K}$ corresponds to a *congruence of $A$.
Thus, if one has canonical forms for $\mathcal{S}$ and $\mathcal{C}$
under simultaneous congruence (often called \emph{canonical pairs}),
then one can obtain a canonical form for $A$ under congruence as a
consequence. Similarly, a canonical form for $A$ under *congruence
can be obtained if one has canonical pairs for two Hermitian
matrices under simultaneous *congruence. Canonical pairs of both
types may be found in Thompson's landmark paper \cite{thom} as well
as in Lancaster and Rodman's recent reviews \cite{L-R} and
\cite{lan-rod}. Thompson's canonical pairs were used to obtain
canonical matrices for congruence over the real field by Lee and
Weinberg \cite{leewei}, who observed that \textquotedblleft the
complex case follows from Thompson's results just as
easily.\textquotedblright

However, deriving canonical forms for complex congruence and
*congruence from canonical pairs is like deriving the theory of
conformal mappings in the real plane from properties of conjugate
pairs of real harmonic functions. It can be done, but there are huge
technical and conceptual advantages to working with complex analytic
functions of a complex variable instead. Likewise, to derive
congruence or *congruence canonical forms for a complex matrix $A$
we advocate working directly with $A$ rather than with its
associated pairs $(\mathcal{S},\mathcal{C})$ or
$(\mathcal{H},\mathcal{K})$. Our approach leads to three simple
canonical block types for complex congruence rather than the six
cumbersome block types found by Lee and Weinberg \cite[p.
208]{leewei}.

Of course, canonical pairs for $(\mathcal{S},\mathcal{C})$ and
$(\mathcal{H},\mathcal{K})$ follow from congruence and
*congruence canonical forms for $A$. Define the following $n$-by-$n$ matrices:
\[
\label{pairh0}M_{n}:=%
\begin{bmatrix}
0 & 1 &  & 0\\
1 & 0 & \ddots & \\
& \ddots & \ddots & 1\\
0 &  & 1 & 0
\end{bmatrix}
,\quad N_{n}:=%
\begin{bmatrix}
0 & 1 &  & 0\\
-1 & 0 & \ddots & \\
& \ddots & \ddots & 1\\
0 &  & -1 & 0
\end{bmatrix}
,
\]%
\[
X_{n}:=%
\begin{bmatrix}
0 &  &  &  &  & (-1)^{n+1}\\
&  &  &  &
\ddd & 0 \\
&  &  & -1 &
\ddd & \\
&  & 1 & 0 &  & \\
& -1 & 0 &  &  & \\
1 & 0 &  &  &  & 0
\end{bmatrix}
,\quad Y_{n}:=%
\begin{bmatrix}
0 &  &  &  &  & 0\\
&  &  &  &
\ddd & (-1)^n\\
&  &  & 0 &
\ddd & \\
&  & 0 & 1 &  & \\
& 0 & -1 &  &  & \\
0 & 1 &  &  &  & 0
\end{bmatrix},
\]
and a 2-parameter version of the matrix (\ref{Deltan})%
\[
\Delta_{n}(a,b):=%
\begin{bmatrix}
0 &  &  & a\\
&  &
\ddd  & b\\
& a & \ddd & \\
a & b &  & 0
\end{bmatrix}
\text{,}\qquad a,b\in\mathbb{C}.
\]

Define the \textit{direct sum} of two matrix pairs
\[
(A_{1},A_{2})\oplus(B_{1},B_{2}):=(A_{1}\oplus B_{1},\,A_{2}\oplus
B_{2})
\]
and the \textit{skew sum} of two matrices
\[
\lbrack A\,\diagdown\,B]:=%
\begin{bmatrix}
0 & B\\
A & 0
\end{bmatrix}
.
\]

Matrix pairs $(A_{1},A_{2})$ and $(B_{1},B_{2})$ are said to be
\emph{simultaneously congruent} (respectively, \emph{simultaneously
*congruent}) if there is a nonsingular matrix $R$ such that $A_{1}=R^{T}%
B_{1}R$ and $A_{2}=R^{T}B_{2}R$ (respectively,
$A_{1}=R^{\ast}B_{1}R$ and $A_{2}=R^{\ast}B_{2}R$). This
transformation is a \emph{simultaneous congruence }(respectively,
\emph{a simultaneous {\rm*\!}congruence}) \emph{of the pair}
$(B_{1},B_{2})$ \emph{via} $R$.

The following theorem lists the canonical pairs that can occur and
their associations with the congruence and
*congruence canonical
matrices listed in \eqref{table1} and \eqref{table2} of Theorem
\ref{t2}. The parameters $\lambda$ and $\mu$ are as described in
Theorem \ref{t2}; the parameters $\nu$, $c$, $a$, and $b$ in the
canonical pairs are functions of $\lambda$ and $\mu$.

\begin{theorem}
\label{ths} (a) Each pair $(\mathcal{S},\mathcal{C})$ consisting
of a symmetric complex matrix $\mathcal{S}$ and a skew-symmetric
complex matrix $\mathcal{C}$ of the same size is simultaneously
congruent to a direct sum of pairs, determined uniquely up to
permutation of summands, of the following three types, each
associated with the indicated congruence canonical matrix type for
$A=\mathcal{S}+
\mathcal{C}$:\medskip%
\newline
\begin{equation}\label{table3}
\renewcommand{\arraystretch}{1.2}
\begin{tabular}
[c]{|c|c|}\hline Type 0: $J_{n}(0)$  & $(M_{n},\,N_{n})$
  \\\hline
Type I: $\Gamma_{n}$
   &
$\left( X_{n},Y_{n}\right)$ if $n$ is odd,
  \\
& $\left( Y_{n},X_{n}\right) $ if  $n$ is even
  \\\hline
Type II: $H_{2n}(\mu)$
& $(  [J_{n}%
(\mu+1)\diagdown J_{n}(\mu+1)^{T}],$\qquad\qquad
\\
& $\qquad\qquad [J_{n}(\mu-1)\diagdown-J_{n}(\mu -1)^{T}])$
\\
\multicolumn{2}{|c|}{$0\neq\mu\neq(-1)^{n+1}$ and $\mu$ is
determined up to replacement by $\mu^{-1}$}
\\ \hline
\end{tabular}
\end{equation}\\[1pt]
The Type II pair in \eqref{table3} can be replaced by two
alternative pairs
\begin{equation}\label{table4}
\renewcommand{\arraystretch}{1.2}
\begin{tabular}
[c]{|c|c|}
   \hline
Type II: $H_{2n}(\mu)$
   &
$\left(  [I_{n}\diagdown I_{n}],\: [J_{n}(\nu)\diagdown
-J_{n}(\nu)^{T}]\right) $
   \\
  $0\neq\mu\neq-1$
   &
$\nu \ne 0$ if $n$ is odd, $\nu\ne\pm 1,$
   \\
$\mu \ne 1$ if $n$ is odd
 & $\nu$ is determined up to
replacement by $-\nu$
   \\
   \hline Type II:
$H_{2n}(-1)$
    &
$\left(  [J_{n}(0)\diagdown J_{n}%
(0)^{T}],\: [I_{n}\diagdown-I_{n}]\right) $
   \\
$n$ is odd & $n$ is odd\\
\hline
\end{tabular}
\end{equation}
in which
\[
\nu=\frac{\mu-1}{\mu+1}.
\]

(b) Each pair $(\mathcal{H},\mathcal{K})$ of Hermitian matrices of
the same size is simultaneously {\rm*\!}congruent to a direct sum of
pairs, determined uniquely up to permutation of summands, of the
following four types, each associated with the indicated congruence
canonical matrix
type for $A=\mathcal{H}+i\mathcal{K}$:\medskip%
\begin{equation}\label{table5}
\renewcommand{\arraystretch}{1.2}
\begin{tabular}
[c]{|c|c|}\hline $\text{Type 0: }J_{n}(0)$
    &
$(M_{n},\,iN_{n})$
    \\\hline
$\text{Type I: }\lambda\Delta_{n}$
     &
$\pm\left( \Delta_{n}(1,0),\: \Delta_{n}(c,1)\right)$
  \\
$|\lambda|=1,\ \lambda^{2}\ne -1$
   &
$c\in\mathbb{R}$
   \\\hline
$\text{Type I: }\lambda\Delta_{n}$
  &
$\pm\left( \Delta _{n}(0,1),\: \Delta_{n}(1,0)\right)$
  \\
$\lambda^{2}=-1$&
   \\\hline
$\text{Type II: }H_{2n}(\mu)$
   &
$\left( [I_{n}\,\diagdown\,I_{n}],\:
[J_{n}(a+ib)\,\diagdown\,J_{n}%
(a+ib)^{\ast}]\right)$
  \\
$|\mu| >1$
  &
$a,b\in{\mathbb R},\ a+bi\ne i,\ b>0$
  \\\hline
\end{tabular}
\end{equation}
in which
\[
c=\frac{\operatorname{Im}\lambda} {\operatorname{Re}\lambda},\qquad
a=\frac{2\operatorname{Im}\mu} {|1+\mu|^{2}},\qquad
b=\frac{|\mu|^{2}-1} {|1+\mu|^{2}}.
\]
\end{theorem}

\section{Congruence}

\label{s.pr1b}

The \emph{cosquare} of a nonsingular complex matrix $A$ is $A^{-T}A$. If $B$
is congruent to $A$, then $A=S^{T}BS$ for some nonsingular $S$ and hence%
\[
A^{-T}A=(S^{T}BS)^{-T}(S^{T}BS)=S^{-1}B^{-T}S^{-T}S^{T}BS=S^{-1}%
(B^{-T}B)S\text{,}%
\]
so \textit{congruent nonsingular matrices have similar cosquares}. The
following lemma establishes the converse assertion; an analogous statement for
arbitrary systems of forms and linear mappings was given in \cite{roi} and
\cite[Theorem 1 and \S \,2]{ser1}.

\begin{lemma}
\label{l_nonsi} Nonsingular complex matrices $A$ and $B$ are congruent if and
only if their cosquares are similar.
\end{lemma}

\begin{proof}
Let $A^{-T}A$ and $B^{-T}B$ be similar via $S$, so that
\begin{equation}
A^{-T}A=S^{-1}B^{-T}BS=\left(  S^{-1}B^{-T}S^{-T}\right)  \left(
S^{T}BS\right)  =C^{-T}C, \label{hbj}%
\end{equation}
in which $C:=S^{T}BS$. It suffices to prove that $C$ is congruent to $A$. Let
$M:=CA^{-1}$ and deduce from \eqref{hbj} that
\[
M=CA^{-1}=C^{T}A^{-T}=\left(  A^{-1}C\right)  ^{T}\text{\ and\ }M^{T}%
=A^{-1}C\text{.}%
\]
Thus,%
\[
C=MA=AM^{T}%
\]
and hence%
\[
q(M)A=Aq(M^{T})=Aq(M)^{T}%
\]
for any polynomial $q(t)$. The theory of primary matrix functions
\cite[Section 6.4]{horn} ensures that there is a polynomial $p(t)$ such that
$p(M)^{2}=M$, so $p(M)$ is nonsingular and%
\[
p(M)A=Ap(M)^{T}\text{.}%
\]
Thus,%
\[
C=MA=p(M)^{2}A=p(M)Ap(M)^{T}%
\]
so $C$ is congruent to $A$ via $p(M)$.
\end{proof}

\begin{proof}
[Proof of Theorem \textrm{\ref{t2}(a)}]Let $A$ be square and nonsingular. The
Jordan Canonical Form of $A^{-T}A$ has a very special structure:
\begin{equation}
\bigoplus_{i=1}^{p}\left(  J_{m_{i}}(\mu_{i})\oplus J_{m_{i}}(\mu_{i}%
^{-1})\right)  \oplus\bigoplus_{j=1}^{q}J_{n_{j}}((-1)^{n_{j}+1})\text{,\quad
}0\neq\mu_{i}\neq(-1)^{m_{i}+1}\text{;} \label{non3}%
\end{equation}
see \cite[Theorem 2.3.1]{Wall} or \cite[Theorem 3.6]{bal}. Using \eqref{non3},
form the matrix
\[
B=\bigoplus_{i=1}^{p}H_{2m_{i}}(\mu_{i})\,\oplus\bigoplus_{j=1}^{q}%
\Gamma_{n_{j}}.
\]
Since the cosquare
\[
H_{2m}(\mu)^{-T}H_{2m}(\mu)=%
\begin{bmatrix}
0 & I_{m}\\
J_{m}(\mu)^{-T} & 0
\end{bmatrix}%
\begin{bmatrix}
0 & I_{m}\\
J_{m}(\mu) & 0
\end{bmatrix}
=%
\begin{bmatrix}
J_{m}(\mu) & 0\\
0 & J_{m}(\mu)^{-T}%
\end{bmatrix}
\]
is similar to $J_{m}(\mu)\oplus J_{m}(\mu^{-1})$ and the cosquare%
\[
\Gamma_{n}^{-T}\Gamma_{n}=(-1)^{n+1}%
\begin{bmatrix}
\vdots & \vdots & \vdots & \vdots &
\ddd \\
-1 & -1 & -1 & -1 & \\
1 & 1 & 1 &  & \\
-1 & -1 &  &  & \\
1 &  &  &  & 0%
\end{bmatrix}
^{T}\!\!\!\!\cdot\Gamma_{n}=(-1)^{n+1}%
\begin{bmatrix}
1 & 2 &  & \star\\
& 1 & \ddots & \\
&  & \ddots & 2\\
0 &  & & 1
\end{bmatrix}
\]
is similar to $J_{n}((-1)^{n+1})$, \eqref{non3} is the Jordan
Canonical Form of $B^{-T}B$. Hence $A$ and $B$ have similar
cosquares. Lemma \ref{l_nonsi} now ensures that $A$ and $B$ are
congruent. Moreover, for any $C$ that is congruent to $A$, its
cosquare $C^{-T}C$ is similar to $A^{-T}A$, so it has the Jordan
Canonical Form \eqref{non3}, which is uniquely determined up to
permutation of summands. Hence $B$ is uniquely determined up to
permutation of its direct summands.
\end{proof}

\medskip Our proof of Theorem \ref{t2}(a) shows that the congruence canonical
form of a square complex matrix $A$ can be constructed as follows:

\begin{enumerate}
\item Use the regularizing algorithm in \textrm{\cite[Section 2]{hor-ser1}} to
construct a regularizing decomposition \eqref{1.1} of $A$ (if desired, one may
use only unitary transformations in that algorithm).

\item Let $B$ be the regular part of $A$ and determine the Jordan canonical
form (\ref{non3}) of its cosquare $B^{-T}B$.

\item Then
\[
\bigoplus_{i=1}^{p}\,H_{2m_{i}}(\mu_{i})\oplus\bigoplus_{j=1}^{q}\Gamma
_{n_{j}}\oplus J_{r_{1}}(0)\oplus\cdots\oplus J_{r_{p}}(0)
\]
is the congruence canonical form of $A$.
\end{enumerate}

\section{*Congruence}

\label{s.pr1x}

The *\emph{cosquare} of a nonsingular complex matrix $A$ is $\mathcal{A}%
=A^{-\ast}A$. If $B$ is *congruent to $A$, then $A=S^{\ast}BS$ for some
nonsingular $S$ and hence%
\[
A^{-\ast}A=(S^{\ast}BS)^{-\ast}(S^{\ast}BS)=S^{-1}B^{-\ast}S^{-\ast}S^{\ast
}BS=S^{-1}(B^{-\ast}B)S\text{,}%
\]
so *congruent nonsingular matrices have similar *cosquares. However, $[-1]$ is
the *cosquare of both $[i]$ and $[-i]$, which are not *congruent: there is no
nonzero complex $s$ such that $-i=\bar{s}is=|s|^{2}i$. Nevertheless, there is
a useful analog of Lemma \ref{l_nonsi} for *congruence. We denote the set of
\textit{distinct} eigenvalues of a square matrix $X$ by $\operatorname{dspec}%
X$.

\begin{lemma}
\label{Lemma*congrence}Let $A$ and $B$ be nonsingular $n$-by-$n$
complex matrices with similar *cosquares, that is,
$A^{-\ast}A=S^{-1}(B^{-\ast}B)S$ for some nonsingular $S$. Let
$B_{S}:=S^{\ast}BS$, let $M:=B_{S}A^{-1}$, and suppose $M$ has $k$
real negative eigenvalues, counted according to their algebraic
multiplicities $(0\leq k\leq n)$. Then:
\newline
(a) $M$ is similar to a real matrix.
\newline
(b) There are square complex matrices $D_{-}$ and $D_{+}$ of size
$k$ and $n-k$, respectively, such that $A$ is *congruent to
$\left(  -D_{-}\right) \oplus D_{+}$ and $B$ is *congruent to
$D_{-}\oplus D_{+}$.
\end{lemma}

\begin{proof}
We have%
\[
A^{-\ast}A=S^{-1}(B^{-\ast}B)S=\left(  S^{-1}B^{-\ast}S^{-\ast}\right)
\left(  S^{\ast}BS\right)  =B_{S}^{-\ast}B_{S}\text{,}%
\]
from which it follows that%
\[
M=B_{S}A^{-1}=B_{S}^{\ast}A^{-\ast}=\left(  A^{-1}B_{S}\right)  ^{\ast}%
\quad\text{and}\quad M^{\ast}=A^{-1}B_{S},%
\]
and hence%
\begin{equation}
B_{S}=MA=AM^{\ast}\text{.} \label{A*B2}%
\end{equation}
Thus, $M^{\ast}=A^{-1}MA$, so $M$ is similar to a real matrix \cite[Theorem
4.1.7]{HJ1}. Its Jordan blocks with nonreal eigenvalues occur in conjugate
pairs, so there is a nonsingular $T$ such that $TMT^{-1}=M_{-}\oplus M_{+}$,
in which the $k$-by-$k$ matrix $M_{-}$ is either absent or has only real
negative eigenvalues; $M_{+}$ has no negative eigenvalues and is similar to a
real matrix if it is present. Moreover, if we partition $TAT^{\ast}=\left[
A_{ij}\right]  _{i,j=1}^{2}$ conformally to $M_{-}\oplus M_{+}$, Sylvester's
Theorem on linear matrix equations \cite[Theorem 4.4.6]{horn} ensures that
\begin{equation}
TAT^{\ast}=A_{11}\oplus A_{22} \label{kyd}%
\end{equation}
since the equalities
$\operatorname{dspec}M_{-}\cap
\operatorname{dspec}M_{+}=\varnothing$, $\operatorname{dspec}M_{-}%
=\operatorname{dspec}M_{-}^{\ast}$, and $\operatorname{dspec}M_{+}%
=\operatorname{dspec}M_{+}^{\ast}$ imply that
\begin{align*}
TB_{S}T^{\ast}  &  =\left(  TMT^{-1}\right)  \left(  TAT^{\ast}\right)
=\left(  TAT^{\ast}\right)  \left(  T^{-\ast}M^{\ast}T^{\ast}\right) \\
&  =\left(  M_{-}\oplus M_{+}\right)  \left(  TAT^{\ast}\right)  =\left(
TAT^{\ast}\right)  \left(  M_{-}^{\ast}\oplus M_{+}^{\ast}\right) \\
&  =\left[
\begin{array}
[c]{cc}%
M_{-}A_{11} & M_{-}A_{12}\\
M_{+}A_{21} & M_{+}A_{22}%
\end{array}
\right]  =\left[
\begin{array}
[c]{cc}%
A_{11}M_{-}^{\ast} & A_{12}M_{+}^{\ast}\\
A_{21}M_{-}^{\ast} & A_{22}M_{+}^{\ast}%
\end{array}
\right] \\
&  =M_{-}A_{11}\oplus M_{+}A_{22}=A_{11}M_{-}^{\ast}\oplus A_{22}M_{+}^{\ast}.
\end{align*}
Thus,%
\[
\left(  -M_{-}\right)  ]A_{11}=A_{11}(-M_{-}^{\ast})\quad\text{and}\quad
M_{+}A_{22}=A_{22}M_{+}^{\ast},%
\]
so%
\[
q_{1}(-M_{-})A_{11}=A_{11}q_{1}(-M_{-}^{\ast})=A_{11}q_{1}(-M_{-})^{\ast}%
\]
and%
\[
q_{2}(M_{+})A_{22}=A_{22}q_{2}(M_{+}^{\ast})=A_{22}q_{2}(M_{+})^{\ast}%
\]
for any polynomials $q_{1}(t)$ and $q_{2}(t)$ with real coefficients. Neither
$-M_{-}$ nor $M_{+}$ has any negative eigenvalues and each is similar to a
real matrix, so \cite[Theorem 2(c)]{HP} ensures that there are polynomials
$g(t)$ and $h(t)$ with real coefficients such that $g(-M_{-})^{2}=-M_{-}$ and
$h(M_{+})^{2}=M_{+}$. It follows that $B$ is *congruent to%
\begin{align*}
TB_{S}T^{\ast}  &  =-\left(  -M_{-}A_{11}\right)  \oplus M_{+}A_{22}%
=-g(M_{-})^{2}A_{11}\oplus h(M_{+})^{2}A_{22}\\
&  =-g(-M_{+})A_{11}g(-M_{+})^{\ast}\oplus h(M_{+})A_{22}h(M_{+})^{\ast}%
=D_{-}\oplus D_{+},%
\end{align*}
in which $D_{-}=-g(-M_{+})A_{11}g(-M_{+})^{\ast}$ and $D_{+}=h(M_{+}%
)A_{22}h(M_{+})^{\ast}$; $A$ is *congruent to $\left(  -D_{-}\right)  \oplus
D_{+}$ by (\ref{kyd}) and $B$ is *congruent to $D_{-}\oplus D_{+}$.
\end{proof}

\medskip The result cited from \cite[Theorem 2(c)]{HP} gives sufficient conditions for $f(X)=Y$
to have a real solution $X$ for a given real $Y$. The key conditions
are that $f$ is analytic and one-to-one on a domain that is
symmetric with respect to the real axis and $f^{-1}$ is typically
real, that is, $f(\bar{z})=\overline{f(z)}$ on the range of $f$.
Under these conditions there is a solution $X$ that is a polynomial
in $Y$ with real coefficients. In the case at hand, $f(z)=z^{2}$ on
the open right half-plane; this special case appears in \cite[p.
545]{Hua} and was employed in \cite[p. 356]{thom} and \cite[Lemma
7.2]{L-R} to study canonical pairs of Hermitian matrices.

Two noteworthy special cases of Lemma \ref{Lemma*congrence} occur
when $M$ either has only positive eigenvalues ($k=n$) or only
negative eigenvalues ($k=0$). In the former case, $A$ is
*congruent to $B$; in the latter case, $A$ is *congruent to $-B$.

\begin{proof}[Proof of Theorem
{\ref{t2}(b)}: Existence] Let
$A$ be nonsingular and
let $\mathcal{A}=A^{-\ast}A$ denote its *cosquare. Because $\mathcal{A}%
^{-\ast}=(A^{-\ast}A)^{-\ast}=AA^{-\ast}=A\mathcal{A}A^{-1}$, for each
eigenvalue $\lambda$ of $\mathcal{A}$ and each $k=1,2,\ldots$, $J_{k}%
(\lambda)$ and $J_{k}(\bar{\lambda}^{-1})$ have equal multiplicities in the
Jordan Canonical Form of $\mathcal{A}$. Since $\lambda=\bar{\lambda}^{-1}$
whenever $|\lambda|=1$, this pairing is trivial for any eigenvalue of
$\mathcal{A}$ that has modulus one; it is nontrivial for eigenvalues
$\mathcal{A}$ whose modulus is greater than one.

Let%
\begin{equation}
\bigoplus_{i=1}^{p}\left(  J_{m_{i}}(\mu_{i})\oplus J_{m_{i}}(\bar{\mu}%
_{i}^{-1})\right)  \oplus\bigoplus_{j=1}^{q}J_{n_{j}}(e^{i\phi_{j}}%
),\quad\left\vert \mu_{i}
\right\vert >1\text{,\ }0\leq\phi_{j}<2\pi\label{JCF}%
\end{equation}
be the Jordan Canonical Form
of $\mathcal{A}$ and use it
to construct the
matrix%
\begin{equation}
B:=\bigoplus_{i=1}^{p}H_{2m_{i}}(\mu_{i})\oplus\bigoplus_{j=1}^{q}e^{i\phi
_{j}/2}\Delta_{n_{j}}\text{.} \label{JCF2}%
\end{equation}
Since the *cosquare
\[
H_{2m}(\mu)^{-\ast}H_{2m}(\mu)=%
\begin{bmatrix}
0 & I_{m}\\
J_{m}(\mu)^{-\ast} & 0
\end{bmatrix}%
\begin{bmatrix}
0 & I_{m}\\
J_{m}(\mu) & 0
\end{bmatrix}
=%
\begin{bmatrix}
J_{m}(\mu) & 0\\
0 & J_{m}(\mu)^{-\ast}%
\end{bmatrix}
\]
is similar to $J_{m}(\mu)\oplus J_{m}(\bar{\mu}^{-1})$ and the *cosquare%
\begin{equation}\label{mdt}
\Delta_{n}^{-\ast}\Delta_{n}=%
\begin{bmatrix}
1 & 2i &  & \star\\
& 1 & \ddots & \\
&  & \ddots & 2i\\
0 & & & 1
\end{bmatrix}
\end{equation}
is similar to $J_{n}(1)$,
(\ref{JCF}) is also the
Jordan Canonical Form of
$B^{-\ast}B$. Hence $A$ and
$B$ have similar *cosquares,
whose common Jordan Canonical
Form contains $2p+q$ Jordan
blocks ($p$ block pairs of
the form $J_{m}(\mu)\oplus
J_{m}(\bar{\mu}^{-1})$ with
$\left\vert \mu\right\vert
>1$ and $q$ blocks of the
form $\lambda\Delta_{n}$ with
$\left\vert \lambda
\right\vert =1$).

Lemma \ref{Lemma*congrence} ensures that there are matrices $D_{+}$ and
$D_{-}$ such that $A$ is *congruent to $(-D_{-})\oplus D_{+}$ and $B$ is
*congruent to $D_{-}\oplus D_{+}$. If $D_{-}$ is absent, then $A$ is
*congruent to $B$; if $D_{+}$ is absent, then $A$ is *congruent to $-B$. In
both of these cases $A$ is *congruent to a direct sum of the form
\begin{equation}
\bigoplus_{i=1}^{p}\varepsilon_{i}H_{2m_{i}}(\mu_{i})\oplus\bigoplus_{j=1}%
^{q}\delta_{j}e^{i\phi_{j}/2}\Delta_{n_{j}},\quad\varepsilon_{i},\delta_{j}%
\in\{-1,+1\}\text{.} \label{luyf}%
\end{equation}
If both direct summands $D_{-}$ and $D_{+}$ are present, then their sizes are
less than the size of $A$. Reasoning by induction, we may assume that each of
$D_{-}$ and $D_{+}$ is *congruent to a direct sum of the form (\ref{luyf}).
Then $A$ is *congruent to a direct sum of the form (\ref{luyf}) as well. We
may take all $\varepsilon_{i}=1$ in (\ref{luyf}) since each $H_{2m}(\mu)$ is
*congruent to $-H_{2m}(\mu)$:%
\[
\left[
\begin{array}
[c]{cc}%
I_{m} & 0\\
0 & -I_{m}%
\end{array}
\right]  \left[
\begin{array}
[c]{cc}%
0 & I_{m}\\
J_{m}(\mu) & 0
\end{array}
\right]  \left[
\begin{array}
[c]{cc}%
I_{m} & 0\\
0 & -I_{m}%
\end{array}
\right]  =-\left[
\begin{array}
[c]{cc}%
0 & I_{m}\\
J_{m}(\mu) & 0
\end{array}
\right]  \text{.}%
\]
\end{proof}

\medskip We have demonstrated that a nonsingular $A$ is *congruent to a direct
sum of Type I and Type II blocks%
\begin{equation}
\bigoplus_{i=1}^{p}H_{2m_{i}}(\mu_{i})\oplus\bigoplus_{j=1}^{q}\delta
_{j}e^{i\phi_{j}/2}\Delta_{n_{j}},\quad\delta_{j}\in\{-1,+1\}\text{,\ }%
\left\vert \mu_{i}\right\vert >1\text{,\ }0\leq\phi_{j}<2\pi,\label{JCF3}%
\end{equation}
in which the sizes $2m_{i}$ and parameters $\mu_{i}$ of the Type II blocks as
well as the sizes $n_{j}$ and squared parameters $(\delta_{j}e^{i\phi_{j}%
/2})^{2}=e^{i\phi_{j}}$ of the Type I blocks are uniquely determined by $A$.
Our reduction algorithm using Lemma \ref{Lemma*congrence} determines a set of
signs $\{\delta_{j}\}$ that gives the desired *congruence of $A$ to
(\ref{JCF3}), but we must show that no other choice of signs is possible: in
the set of Type I blocks in (\ref{JCF3}) with equal sizes $n_{j}$ and equal
coefficients $e^{i\phi_{j}/2}$ the number of blocks with sign equal to $+1$
(and hence also the number of blocks with signs equal to $-1$) is uniquely
determined by $A$.\medskip\

\begin{proof}
[Proof of Theorem \textrm{\ref{t2}(b): Uniqueness}]Let each of $A$ and $B$ be
a direct sum of Type I and Type II blocks and suppose that $A$ and $B$ are
*congruent. We have shown that $A$ and $B$ have the form%
\[
A=\bigoplus_{i=1}^{p}H_{2m_{i}}(\mu_{i})\oplus\bigoplus_{j=1}^{q}\lambda
_{j}\Delta_{n_{j}},\quad B=\bigoplus_{i=1}^{p}H_{2m_{i}}(\mu_{i}%
)\oplus\bigoplus_{j=1}^{q}\kappa_{j}\lambda_{j}\Delta_{n_{j}},%
\]
in which all $\kappa_{j}\in\{-1,+1\}$, all\ $\left\vert \mu_{i}\right\vert
>1$,\ and all $\left\vert \lambda_{j}\right\vert =1$. Our goal is to prove
that each of these direct sums may be obtained from the other by a permutation
of summands. We may rearrange the summands to present $A=A_{1}\oplus A_{2}$
and $B=B_{1}\oplus B_{2}$, in which%
\[
A_{1}=\bigoplus_{r=1}^{k}\lambda_{r}\Delta_{n_{r}},\quad B_{1}=\bigoplus
_{r=1}^{k}\kappa_{r}\lambda_{r}\Delta_{n_{r}},%
\]
and $\lambda_{1}^{2}=\cdots=\lambda_{k}^{2}\neq\lambda_{\ell}^{2}$ for all
$\ell=k+1,\ldots,q$. Let $S^{\ast}AS=B$ and partition $S=[S_{ij}]_{i,j=1}^{2}$
conformally with $A_{1}\oplus A_{2}$. Since the *cosquares of $A$ and $B$ are
similar via $S$, we have $S\left(  B^{-\ast}B\right)  =\left(  A^{-\ast
}A\right)  S$ and hence%
\[
\left[
\begin{array}
[c]{cc}%
S_{11}\left(  B_{1}^{-\ast}B_{1}\right)  & S_{12}\left(  B_{2}^{-\ast}%
B_{2}\right) \\
S_{21}\left(  B_{1}^{-\ast}B_{1}\right)  & S_{22}\left(  B_{2}^{-\ast}%
B_{2}\right)
\end{array}
\right]  =\left[
\begin{array}
[c]{cc}%
\left(  A_{1}^{-\ast}A_{1}\right)  S_{11} & \left(  A_{1}^{-\ast}A_{1}\right)
S_{12}\\
\left(  A_{2}^{-\ast}A_{2}\right)  S_{21} & \left(  A_{2}^{-\ast}A_{2}\right)
S_{22}%
\end{array}
\right]  \text{.}%
\]
But
\[\operatorname{dspec}\left(  B_{2}^{-\ast}B_{2}\right)  \cap
\operatorname{dspec}\left(  A_{1}^{-\ast}A_{1}\right) =\varnothing
\]
and
\[\operatorname{dspec}\left( B_{1}^{-\ast}B_{1}\right)  \cap
\operatorname{dspec}\left( A_{2}^{-\ast}A_{2}\right) =\varnothing,
\]
so Sylvester's Theorem on linear matrix equations ensures that
$S=S_{11}\oplus S_{22}$. Thus, $A_{1}$ is *congruent to $B_{1}$
via $S_{11}$ and hence it suffices to consider the case $A=A_{1}$
and $B=B_{1}$. Moreover, dividing both $A$ and $B$ by
$\lambda_{1}$ it suffices to consider a pair of *congruent
matrices of the form%
\begin{equation}
A=\bigoplus_{r=1}^{k}\varepsilon_{r}\Delta_{n_{r}}\text{,}\;B=\bigoplus
_{r=1}^{k}\delta_{r}\Delta_{n_{r}},\quad\varepsilon_{r},\delta_{r}%
\in\{-1,+1\}\text{.} \label{U1}%
\end{equation}

We may assume that the summands in (\ref{U1}) are arranged so that $1\leq
n_{1}\leq\cdots\leq n_{k}$. Define%
\begin{equation}
N:=(n_{1},\dots,n_{k})\quad\text{and}\quad|N|:=n_{1}+\cdots+n_{k}\text{.}
\label{a20}%
\end{equation}
Let%
\begin{equation}
J_{N}:=J_{n_{1}}(0)\oplus\cdots\oplus J_{n_{k}}(0)\text{,} \label{njh}%
\end{equation}
denote%
\[
\Delta_{N}:=\Delta_{n_{1}}\oplus\cdots\oplus\Delta_{n_{k}} \quad
\text{and}\quad\mathcal{D}_{N}:=\Delta_{N}^{-\ast}\Delta_{N}\text{,}%
\]
and let
\[
P_{N}:=P_{n_{1}}\oplus\cdots\oplus P_{n_{k}}\text{,}%
\]
in which%
\[
P_{n}:=%
\begin{bmatrix}
0 &  & 1\\
&
\ddd & \\
1 &  & 0
\end{bmatrix}
\]
is the $n$-by-$n$ \emph{reversal matrix}.

Since $P_{N}\Delta_{N}=I+iJ_{N}$ and $J_{N}$ is nilpotent,%
\begin{align*}
\mathcal{D}_{N}  &  =\Delta_{N}^{-\ast}\Delta_{N}=\left(  P_{N}\Delta
_{N}^{\ast}\right)  ^{-1}P_{N}\Delta_{N}=(I-iJ_{N})^{-1}(I+iJ_{N})\\
&  =(I+iJ_{N}+i^{2}J_{N}^{2}+i^{3}J_{N}^{3}+\cdots)(I+iJ_{N})\\
&  =I+2iJ_{N}+2i^{2}J_{N}^{2}+2i^{3}J_{N}^{3}+\cdots
\end{align*}
is a polynomial in $J_{N}$. Moreover, $J_{N}=i(I-\mathcal{D}_{N}%
)(I+\mathcal{D}_{N})^{-1}$ and $(I+\mathcal{D}_{N})^{-1}$ is a polynomial in
$\mathcal{D}_{N}$, so $J_{N}$ is a polynomial in $\mathcal{D}_{N}$.

Let $C$ be $|N|$-by-$|N|$ and partition $C=[C_{ij}]_{i,j=1}^{k}$ conformally
to $J_{N}$, so each block $C_{ij}$ is $n_{i}$-by-$n_{j}$. Then $C$ commutes
with $J_{N}$ if and only if each block $C_{ij}$ has the form
\begin{equation}
C_{ij}=%
\begin{cases}%
\begin{bmatrix}
&  & c_{ij} & c_{ij}^{(2)} & \dots & c_{ij}^{(n_{i})}\\
&  &  & c_{ij} & \ddots & \vdots\\
&  &  &  & \ddots & c_{ij}^{(2)}\\
0 &  &  &  &  & c_{ij}%
\end{bmatrix}
& \text{if $i\leq j$,}\\
& \\[-4mm]%
\begin{bmatrix}
c_{ij} & c_{ij}^{(2)} & \dots & c_{ij}^{(n_{j})}\\
& c_{ij} & \ddots & \vdots\\
&  & \ddots & c_{ij}^{(2)}\\
&  &  & c_{ij}\\
&  &  & \\
0 &  &  &
\end{bmatrix}
& \text{if $i>j$;}%
\end{cases}
\label{a21}%
\end{equation}
see \cite[Section VIII, \S \,2]{gan} or \cite[Lemma 4.4.11]{horn}. Each
diagonal block $C_{ii}$ is upper Toeplitz; each block $C_{ij}$ with $i<j$ has
an upper Toeplitz submatrix that is preceded by a zero block; each block
$C_{ij}$ with $i>j$ has an upper Toeplitz submatrix with a zero block below
it. An $|N|$-by-$|N|$ matrix whose blocks have the form (\ref{a21}) is said to
be $N$\textit{-upper Toeplitz}.

If $C$ is $N$-upper Toeplitz, then so is $P_{N}C^{\ast}P_{N}$; the upper
Toeplitz submatrix in its $i,j$ block is the complex conjugate of the upper
Toeplitz submatrix in the $j,i$ block of $C$. The matrices $\mathcal{D}_{N}$,
$P_{N}A$, and $P_{N}B$ (see (\ref{U1})) are all $N$-upper Toeplitz.

Let $C=[C_{ij}]_{i,j=1}^{k}be$ a given $N$-upper Toeplitz matrix. Consider the
mapping $C\mapsto\underline{C}$ that takes $C$ into the $k\times k$ matrix
whose $i,j$ entry is $c_{ij}$ if $n_{i}=n_{j}$ and is $0$ otherwise; see
(\ref{a21}). Partition the entries of $N$ (already nondecreasingly ordered)
into groups of equal entries
\[
1\leq n_{1}=\cdots=n_{r}<n_{r+1}=\cdots=n_{l}<n_{l+1}=\cdots=n_{t}<\cdots
\]
and observe that $\underline{C}$ is structurally block diagonal:%
\begin{equation}
\underline{C}=C_{1}\oplus C_{2}\oplus C_{3}\oplus\cdots\label{lok}%
\end{equation}
The sizes of the direct summands of $\underline{C}$ are the multiplicities of
the entries of $N$, that is,
\begin{equation}
C_{1}:=%
\begin{bmatrix}
c_{11} & \dots & c_{1r}\\
\vdots & \ddots & \vdots\\
c_{r1} & \dots & c_{rr}%
\end{bmatrix}
,\quad C_{2}:=%
\begin{bmatrix}
c_{r+1,r+1} & \dots & c_{r+1,\ell}\\
\vdots & \ddots & \vdots\\
c_{\ell,r+1} & \dots & c_{\ell\ell}%
\end{bmatrix}
,\,\dots\label{mkg}%
\end{equation}
In addition, for any $N$-upper Toeplitz matrix $D$ we have $\underline
{CD}=\underline{C}\,\cdot\,\underline{D}$. If $C$ is nonsingular, we have
$I=\underline{CC^{-1}}=\underline{C}\,\cdot\,\underline{C^{-1}}$, which
implies that $\underline{C}$ is nonsingular. A computation reveals that
$\underline{P_{N}C^{\ast}P_{N}}=(\underline{C})^{\ast}$.

Finally, consider the *congruent matrices $A$ and $B$ in (\ref{U1}), which
satisfy $A^{-\ast}A=B^{-\ast}B=\mathcal{D}_{N}$. We have $\underline{P_{N}%
A}=\operatorname{diag}(\varepsilon_{1},\ldots,\varepsilon_{k})$ and
$\underline{P_{N}B}=\operatorname{diag}(\delta_{1},\ldots,\delta_{k})$. Let
$S^{\ast}AS=B$, so $B^{-\ast}B=S^{-1}\left(  A^{-\ast}A\right)  S$ and hence
\[
S(B^{-\ast}B)=S\mathcal{D}_{N}=\mathcal{D}_{N}S=\left(  A^{-\ast}A\right)
S\text{.}%
\]
Since $S$ commutes with $\mathcal{D}_{N}$ and $J_{N}$ is a polynomial in
$\mathcal{D}_{N}$, $S$ commutes with $J_{N}$ and hence is $N$-upper Toeplitz.
Moreover,
\[
P_{N}B=P_{N}S^{\ast}AS=\left(  P_{N}S^{\ast}P_{N}\right)  \left(
P_{N}A\right)  S
\]
and each of the matrices $P_{N}B$, $P_{N}A$, $S$, and $P_{N}S^{\ast}P_{N}$ is
$N$-upper Toeplitz. Therefore,%
\[
\underline{P_{N}B}=\left(  \underline{P_{N}S^{\ast}P_{N}}\right)  \left(
\underline{P_{N}A}\right)  \underline{S}\text{,}%
\]
that is,%
\[
\operatorname{diag}(\delta_{1},\ldots,\delta_{k})=\underline{S}^{\ast
}\operatorname{diag}(\varepsilon_{1},\ldots,\varepsilon_{k})\underline
{S}\text{.}%
\]
But $\underline{S}$ is nonsingular and block diagonal:
\[
\underline{S}=S_{1}\oplus S_{2}\oplus S_{3}\oplus\cdots,
\]
in which the respective nonsingular blocks $S_{1},S_{2},...$ are the same size
as the respective blocks $C_{1},C_{2},...$ in (\ref{mkg}). Thus,
$\operatorname{diag}(\delta_{1},\ldots,\delta_{r})=S_{1}^{\ast}%
\operatorname{diag}(\varepsilon_{1},\ldots,\varepsilon_{r})S_{1},$
$\operatorname{diag}(\delta_{r+1},\ldots,\delta_{\ell})=S_{2}^{\ast
}\operatorname{diag}(\varepsilon_{r+1},\ldots,\varepsilon_{\ell})S_{2}$, etc.
Sylvester's Inertia Theorem ensures that $\operatorname{diag}(\delta
_{1},\ldots,\delta_{r})$ can be obtained from $\operatorname{diag}%
(\varepsilon_{1},\ldots,\varepsilon_{r})$ by a permutation of its diagonal
entries, $\operatorname{diag}(\delta_{r+1},\ldots,\delta_{\ell})$ can be
obtained from $\operatorname{diag}(\varepsilon_{r+1},\ldots,\varepsilon_{\ell
})$ by a permutation of its diagonal entries, etc. Therefore, each of the
direct sums (\ref{U1}) can be obtained from the other by a permutation of summands.
\end{proof}

\medskip The argument that we have just made also clarifies the final
assertion in Theorem \ref{t2}: A *cosquare $F_{n}^{-\ast}F_{n}$ is similar to
$J_{n}(\lambda)$ with $|\lambda|=1$ if and only if it is not decomposable into
a nontrivial direct sum under similarity.

\section{An alternative algorithm for *congruence\label{Alternative}}

Although we can now determine the *congruence canonical form of a nonsingular
complex matrix $A$, in practice it is useful to have an alternative algorithm.

Let $\mathcal{A}=A^{-\ast}A$, let $\mu_{1},\ldots,\mu_{r}$ be the distinct
eigenvalues of $\mathcal{A}$ with modulus greater than one, and let
$\lambda_{1},\ldots,\lambda_{s}$ be the distinct eigenvalues of $\mathcal{A}$
with modulus one. Let $S$ be any nonsingular matrix such that
\begin{equation}
A^{-\ast}A=S(C_{1}\oplus\cdots\oplus C_{r}\oplus C_{r+1}\oplus\cdots\oplus
C_{r+s})S^{-1}\text{,} \label{alg1}%
\end{equation}
in which $\operatorname{dspec}C_{i}=\{\mu_{i},\bar{\mu}_{i}^{-1}\}$ for
$i=1,\ldots,r$ and $\operatorname{dspec}C_{r+i}=\{\lambda_{i}\}$ for
$i=1,\ldots,s$. One way to achieve this decomposition is to group together
blocks from the Jordan Canonical Form of $\mathcal{A}$, but other strategies
may be employed. Partition $S^{\ast}AS=[A_{ij}]_{i,j=1}^{r+s}$ conformally to
the direct sum in (\ref{alg1}). The argument in the proof of uniqueness in
Section \ref{s.pr1x}\ shows that $S^{\ast}AS$ is block diagonal:%
\begin{equation}
S^{\ast}AS=A_{11}\oplus\cdots\oplus A_{rr}\oplus A_{r+1,r+1}\oplus\cdots\oplus
A_{s+1,s+1}\text{,} \label{alg2}%
\end{equation}
in which each $A_{ii}$ is the same size as $C_{i}$, $\operatorname{dspec}%
A_{ii}^{-\ast}A_{ii}=\{\mu_{i},\bar{\mu}_{i}^{-1}\}$ for $i=1,\ldots,r$, and
$\operatorname{dspec}A_{ii}^{-\ast}A_{ii}=\{\lambda_{i-r}\}$ for
$i=r+1,\ldots,r+s$.

The Type II *congruence blocks are now easy to determine: to each pair of
Jordan blocks $J_{m}(\mu_{i})\oplus J_{m}(\bar{\mu}_{i}^{-1})$ of
$A_{ii}^{-\ast}A_{ii}$ corresponds one Type II *congruence block $H_{2m}%
(\mu_{i})$ of $A$.

Now consider each diagonal block $A_{r+j,r+j}$ in turn. Its *cosquare has a
single eigenvalue $\lambda_{j}=e^{i\phi_{j}}$, $0\leq\phi_{j}<2\pi$. Let the
Jordan Canonical Form of the *cosquare of $e^{-i\phi_{j}/2}A_{r+j,r+j}$ be
$I+J_{N}$, in which $N:=(n_{1},\dots,n_{k})$ and $1\leq n_{1}\leq\cdots\leq
n_{k}$. That *cosquare is similar to $\mathcal{D}_{N}:=\Delta_{N}^{-\ast
}\Delta_{N}$; let $S$ be nonsingular and such that $S\mathcal{D}_{N}%
S^{-1}=e^{-i\phi_{j}}A_{j+r,j+r}^{-\ast}A_{j+r,j+r}$. For notational
convenience, normalize and set $A:=e^{-i\phi_{j}/2}S^{\ast}A_{j+r,j+r}S$.
Then
\[
A^{-\ast}A=e^{-i\phi_{j}}S^{-1}A_{j+r,j+r}^{-\ast}A_{j+r,j+r}S=\mathcal{D}%
_{N}=\Delta_{N}^{-\ast}\Delta_{N},
\]
which implies that
$A=A^{\ast}\Delta_{N}^{-\ast}
\Delta_{N}$ and hence
\begin{equation}
 \label{iyj}
\begin{aligned}
B  &  :=\Delta_{N}^{-1}A=
\Delta_{N}^{-1}(A^{\ast}
\Delta_{N}^{-\ast}\Delta_{N})
=\Delta_{N}^{-1}
(\Delta_{N}^{-1}A)^{\ast
}\Delta_{N}\\
&
=\Delta_{N}^{-1}B^{\ast}\Delta_{N}=\Delta_{N}^{-1}\left(
\Delta_{N}^{\ast
}B\Delta_{N}^{-\ast}\right)  \Delta_{N}=\mathcal{D}_{N}^{-1}B\mathcal{D}%
_{N}\text{.}
\end{aligned}
\end{equation}
Thus, $\mathcal{D}_{N}$
commutes with $B$, so $B$ is
$N$-upper Toeplitz.

Invoking the identity
\[
B^{\ast}=\Delta_{N}B\Delta_{N}^{-1}\text{,}%
\]
already employed in (\ref{iyj}), compute
\[
P_{N}B^{\ast}P_{N}=P_{N}\left(  \Delta_{N}B\Delta_{N}^{-1}\right)
P_{N}=\left(  P_{N}\Delta_{N}\right)  B\Delta_{N}^{-1}P_{N}.
\]
Since $P_{N}\Delta_{N}=I+iJ_{N}$ commutes with any $N$-upper Toeplitz matrix,
it follows that
\begin{equation}
P_{N}B^{\ast}P_{N}=B\left(  P_{N}\Delta_{N}\right)  \Delta_{N}^{-1}%
P_{N}=BP_{N}^{2}=B\text{.} \label{boxstardef}%
\end{equation}

Let $X = [X_{ij}]$ be a given $N$-upper Toeplitz matrix,
partitioned as in (\ref{a21}). The \emph{$N$-block star} of $X$ is
the $N$-upper Toeplitz matrix
\[
X^{\maltese}:=P_{N}X^{\ast }P_{N}.
\]
The upper Toeplitz submatrix of each $i,j$ block of $X^{\maltese}$
is the complex conjugate of the upper Toeplitz submatrix of
$X_{ji}$. We say that $X$ is \emph{$N$-Hermitian} if
$X^{\maltese}=X$, in which case the upper Toeplitz submatrices of
each pair of blocks $X_{ij}$ and $X_{ji}$ are complex conjugates.
If $N=(1,1,\ldots,1)$ then $X^{\maltese}=X^*$ and  $X$ is
$N$-Hermitian if and only if it is Hermitian.

The identity (\ref{boxstardef}) asserts that $B$ is $N$-Hermitian.

If $X$ and $Y$ are $N$-upper Toeplitz, one checks that%
\[
\left(  XY\right)  ^{\maltese}=Y^{\maltese}X^{\maltese}\text{.}%
\]
We say that $N$-upper Toeplitz matrices $X$ and $Y$ are \emph{$^{\maltese\!}%
$congruent} (\emph{$N$-block star congruent}) if there exists a
nonsingular $N$-upper Toeplitz matrix $S$ such that
$S^{\maltese}BS=C$; $^{\maltese\!}$congruence is an equivalence
relation on the set of $N$-upper Toeplitz matrices.

Since $B$ is $N$-upper Toeplitz and $N$-Hermitian, for any
$N$-upper Toeplitz matrix $S$ we have%
\begin{align}
S^{\ast}AS  &  =S^{\ast}\left(  \Delta_{N}B\right)  S=\left(  P_{N}%
S^{\maltese}P_{N}\right)  \Delta_{N}BS=P_{N}S^{\maltese}(P_{N}\Delta
_{N})BS\nonumber\\
&  =P_{N}\left(  P_{N}\Delta_{N}\right)  S^{\maltese}BS=\Delta_{N}\left(
S^{\maltese}BS\right)  \text{.} \label{jrt}%
\end{align}
If we can find a nonsingular $N$-upper Toeplitz $S$ such that
\[
S^{\maltese}BS=\varepsilon_{1}I_{n_{1}}\oplus\cdots\oplus\varepsilon
_{k}I_{n_{k}}\quad\text{with\ }\varepsilon_{i}=\pm1\text{,}%
\]
it follows from \eqref{jrt} that $A$ is *congruent to
\begin{equation}
\varepsilon_{1}\Delta_{n_{1}}\oplus\cdots\oplus\varepsilon_{k}\Delta_{n_{k}}.
\label{signs}%
\end{equation}
Theorem \ref{t2}(b) ensures
that for each $n=1,2,\ldots$
there is a \textit{unique}
set of signs associated with
the blocks $\Delta_{n}$ of
size $n$ in (\ref{signs}).
The following generalization
of Sylvester's Inertia
Theorem provides a way to
construct these signs.

\begin{lemma}
\label{lemHermitian} Let $C$ be nonsingular, $N$-upper Toeplitz,
and {$N$-Hermitian.} { Then there is a nonsingular }$N$-upper
Toeplitz matrix $S$ such that
\begin{equation}
S^{\maltese\!}CS=\varepsilon_{1}I_{n_{1}}\oplus\cdots\oplus\varepsilon
_{k}I_{n_{k}},\qquad\text{each }\varepsilon_{i}\in\{-1,1\}\text{.}
\label{BBdiag}%
\end{equation}

\end{lemma}

\begin{proof}
Since $C$ is nonsingular, $\underline{C}$ and hence all of the direct summands
in (\ref{lok}) are nonsingular as well.

\medskip\textbf{Step 1.} If $c_{11}\neq0$, proceed to Step 2. If $c_{11}=0$,
then $c_{1j}\neq0$ for some $j\leq r$ since $C_{1}$ is nonsingular. Let
$S_{\theta}=[S_{ij}]_{i,j=1}^{k}$ be the $N$-upper Toeplitz matrix in which
the diagonal blocks are identity matrices and all the other blocks are zero
except for $S_{j1}:=e^{i\theta}I_{n_{1}}$ for a real $\theta$ to be
determined. The $1,1$ entry of $S_{\theta}^{\maltese}CS_{\theta}$ is
\begin{equation}
e^{i\theta}c_{1j}+e^{-i\theta}\bar{c}_{1j}+e^{i\theta}e^{-i\theta}%
c_{jj}=2\operatorname{Re}(e^{i\theta}c_{1j})+c_{jj} \label{new_entry}%
\end{equation}
($c_{j1}=\bar{c}_{1j}$ and $c_{jj}=\bar{c}_{jj}$ since $C^{\maltese}=C$).
Choose any $\theta$ for which (\ref{new_entry}) is nonzero.

\medskip\textbf{Step 2.} We may now assume that $c_{11}$ is a nonzero real
number. Let $a:=|c_{11}|^{-1/2}$, so $a^{2}c_{11}=\pm1$. Define the real
$N$-upper Toeplitz matrix $S=aI$ and form $S^{\maltese}CS=SCS$, whose $1,1$
entry is $\pm1$.

\medskip\textbf{Step 3. }We may now assume that $\operatorname{dspec}%
C_{11}=\{c_{11}\}=\{\pm1\}$. Then $\operatorname{dspec}(c_{11}C_{11}%
^{-1})=\{1\}$, so there is a polynomial $p(t)$ with real coefficients such
that $p(C_{11})^{2}=c_{11}C_{11}^{-1}$ \cite[Theorem 6.4.14]{horn},
$p(C_{11})$ is upper Toeplitz and commutes with $C_{11}$, and $p(C_{11}%
)C_{11}p(C_{11})=c_{11}I=\pm I_{n_{1}}$. Define the real $N$-upper Toeplitz
matrix $S=p(C_{11})\oplus I_{n_{2}}\oplus\cdots\oplus I_{n_{k}}$ and form
$S^{\maltese}CS$, whose $1,1$ block is $\pm I_{n_{1}}$.

\medskip\textbf{Step 4.} We may now assume that $C_{11}=\pm I_{n_{1}}$. Define
the $N$-upper Toeplitz matrix%
\[
S=%
\begin{bmatrix}
I_{n_{1}} & -c_{11}C_{12} & \dots & -c_{11}C_{1n_{k}}\\
& I_{n_{2}} & \dots & 0\\
&  & \ddots & \vdots\\
0 &  &  & I_{n_{k}}%
\end{bmatrix}
\]
and form $S^{\maltese}CS$; its $1,j$ and $j,1$ blocks are zero for all
$j=2,\ldots,n_{k}$.

The preceding four steps reduce $C$ by $^{\maltese}$congruence to the form
$\pm I_{n_{1}}\oplus C^{\prime}$. Now reduce $C^{\prime}$ in the same way and
continue. After $k$ iterations of this process we obtain a real diagonal
matrix
\begin{equation}
\varepsilon_{1}I_{n_{1}}\oplus\cdots\oplus\varepsilon_{k}I_{n_{k}}%
,\quad\varepsilon_{i}\in\{-1,1\} \label{signs2}%
\end{equation}
that is $^{\maltese}$congruent to the original matrix $C$.
\end{proof}

\medskip Thus, to determine the *congruence canonical form of a given square
complex matrix $A$, one may proceed as follows:

\begin{enumerate}
\item Apply the regularization algorithm \cite{hor-ser1} to determine the
singular *congruence blocks and the regular part. This reduces the problem to
consideration of a nonsingular $A$.

\item Let $S$ be any nonsingular matrix that gives a similarity between the
*cosquare of $A$ and the direct sum in (\ref{alg1}) and calculate $S^{\ast}%
AS$, which has the block diagonal form (\ref{alg2}). Consider each of these
diagonal blocks in turn.

\item Determine the Type II *congruence blocks by examining the Jordan
Canonical Forms of the diagonal blocks whose *cosquare has a two-point
spectrum: one Type II block $H_{2m}(\mu_{i})$ corresponds to each pair
$J_{m}(\mu_{i})\oplus J_{m}(\bar{\mu}_{i}^{-1})$ in the Jordan Canonical Form
of $A_{ii}^{-\ast}A_{ii}$.

\item For each diagonal block $A_{jj}$ whose *cosquare has a one-point
spectrum (suppose it is $e^{i\phi}$), consider $\hat{A}_{jj}=e^{-i\phi
/2}A_{jj}$. Find an $S$ such that $\hat{A}_{jj}^{-\ast}\hat{A}_{jj}%
=S^{-1}(\Delta_{N}^{-1}\Delta_{N})S$ and consider
$B=\Delta_{N}^{-1}S^{\ast }\hat{A}_{jj}S$, which is $N$-upper
Toeplitz and {$N$-Hermitian.}

\item Use Lemma \ref{lemHermitian} (or some other means) to reduce $B$ by
$^{\maltese\!}$congruence to a diagonal form (\ref{signs2}). Then $\hat
{A}_{jj}$ is *congruent to a direct sum of the form (\ref{signs}) and the
diagonal block $A_{jj}$ corresponds to a direct sum $\varepsilon_{1}%
e^{i\phi/2}\Delta_{n_{1}}\oplus\cdots\oplus\varepsilon_{k}e^{i\phi/2}%
\Delta_{n_{k}}$ of Type I blocks.
\end{enumerate}

\section{Some special *congruences\label{Special}}

A square complex matrix $A$ is diagonalizable by *congruence if and only if
the *congruence canonical form of $A$ contains only 1-by-1 blocks (which can
only be Type 0 blocks $J_{1}(0)=[0]$ and Type I blocks $\lambda\Delta
_{1}=[\lambda]$ with $|\lambda|=1$). The Type 0 *congruence canonical blocks
for $A$ are all 1-by-1 if and only if $A$ and $A^{\ast}$ have the same null
space; an equivalent condition is that there is a unitary $U$ such that
\begin{equation}
A=U^{\ast}(B\oplus0_{k})U\text{ and }B\text{ is nonsingular.}
\label{normalnullspace}%
\end{equation}
The *congruence class of the regular part $B$ is uniquely
determined, so the similarity class of its *cosquare $B^{-\ast}B$
is also uniquely determined. There are no Type II blocks for $A$
and its Type I blocks are all 1-by-1 if and only if the *cosquare
of its regular part is diagonalizable and has only eigenvalues
with unit modulus. Thus, $A$ \textit{is diagonalizable by
*congruence if and only if both of the following conditions are
satisfied: (a) }$A$\textit{ and }$A^{\ast}$\textit{ have the same
null space, and (b) the *cosquare of the regular part of $A$ is
diagonalizable and all its eigenvalues have unit modulus.}

Let $A$ be nonsingular and suppose that its *cosquare $\mathcal{A}$ is
diagonalizable and all its eigenvalues have unit modulus. Let $S$ be any
nonsingular matrix that diagonalizes $\mathcal{A}$, and consider the forms
that (\ref{alg1}) and (\ref{alg2}) take in this case:%
\[
A^{-\ast}A=S\left(  C_{1}\oplus\cdots\oplus C_{q}\right)  S^{-1}%
\]
and%
\[
S^{\ast}AS=E_{1}\oplus\cdots\oplus E_{q},%
\]
in which $C_{j}=e^{2i\theta_{j}}I$ for each $j=1,...,q$, $0\leq\theta
_{1}<\cdots<\theta_{q}<\pi$, each $E_{j}$ is the same size as $C_{j}$, and
each $E_{j}^{-\ast}E_{j}=e^{2i\theta_{j}}I$. We have $e^{-i\theta_{j}}%
E_{j}=e^{i\theta_{j}}E_{j}^{\ast}=(e^{-i\theta_{j}}E_{j})^{\ast}$, so each
matrix $e^{-i\theta_{j}}E_{j}$ is Hermitian. Sylvester's Inertia Theorem
ensures that $e^{-i\theta_{j}}E_{j}$ is *congruent to $I_{n_{j}^{+}}%
\oplus(-I_{n_{j}^{-}})$ for nonnegative integers $n_{j}^{+}$ and $n_{j}^{-}$
that are determined uniquely by (the *congruence class of) $e^{-i\theta_{j}%
}E_{j}$. It follows that each $E_{j}$ is *congruent to
\[
e^{i\theta_{j}}I_{n_{j}^{+}}\oplus(-e^{i\theta_{j}}I_{n_{j}^{-}}%
)=e^{i\theta_{j}}I_{n_{j}^{+}}\oplus(e^{i(\theta_{j}+\pi)}I_{n_{j}^{-}%
})\text{,}%
\]
so $A$ is *congruent to the uniquely determined canonical form%
\[
\left(  e^{i\theta_{1}}I_{n_{1}^{+}}\oplus(e^{i(\theta_{1}+\pi)}I_{n_{1}^{-}%
})\right)  \oplus\cdots\oplus\left(  e^{i\theta_{q}}I_{n_{q}^{+}}%
\oplus(e^{i(\theta_{q}+\pi)}I_{n_{q}^{-}})\right)  \text{,\ }%
\]
$0\leq\theta_{1}<\cdots<\theta_{q}<\pi$.
The angles $\theta_{j}$ for
which $n_{j}^{+}\geq1$
together with the angles
$\theta_{j}+\pi$ for which
$n_{j}^{-}\geq1$ (all
$\theta_{j}\in\lbrack0,\pi)$,
$j=1,\dots,q$) are the
canonical angles of order one
of $A$; the corresponding
integers $n_{j}^{+}$ and
$n_{j}^{-}$ are their
respective multiplicities.

We have just described how to determine the signs
$\varepsilon_{j}$ that occur in (\ref{signs}) when
$N=(1,\ldots,1)$; in this special case, every $|N|$-by-$|N|$
matrix is $N$-upper Toeplitz, and $N$-Hermitian matrices are just
ordinary Hermitian matrices. \textit{Two square complex matrices
of the same size that are diagonalizable by *congruence are
*congruent if and only if they have the same canonical angles of
order one with the same multiplicities.}

This observation is a special case of a more general fact: over the reals or
complexes, each system of forms and linear mappings decomposes uniquely into a
direct sum of indecomposables, up to isomorphism of summands \cite[Theorem
2]{ser1}. This special case was rediscovered in \cite{JF}, which established
uniqueness of the canonical angles and their multiplicities but did not
determine them (the signs $\varepsilon_{j}$ remained ambiguous) except in
special circumstances, e.g., if the field of values of $e^{i\phi}A$ lies in
the open right half plane for some $\phi\in\lbrack0,2\pi)$. Using
(\ref{normalnullspace}) to introduce generalized inverses and a natural
generalized *cosquare, \cite{robinson} later gave an alternative approach that
fully determined the canonical angles and their multiplicities for a complex
matrix that is diagonalizable by *congruence.

If $A$ is normal, it is unitarily *congruent to $\Lambda\oplus0_{k}$, in which
$\Lambda$ is a diagonal matrix whose diagonal entries are the nonzero
eigenvalues of $A$ (including multiplicities). The canonical angles (of order
one; no higher orders occur) of a normal matrix are just the principal values
of the arguments of its nonzero eigenvalues; the multiplicity of a canonical
angle is the number of eigenvalues on the open ray that it determines. Thus,
\textit{two normal complex matrices of the same size are *congruent if and
only if they have the same number of eigenvalues on each open ray from the
origin.} This special case was completely analyzed in \cite{Ikramov}. If $A$
is Hermitian, of course, its nonzero eigenvalues are on only two open rays
from the origin: the positive half-line and the negative half-line. For
Hermitian matrices, the criterion for *congruence of normal matrices is just
Sylvester's Inertia Theorem.

Finally, suppose that the *congruence canonical form of a nonsingular
$n$-by-$n$ complex matrix $A$ has only one block; the singular case is
analyzed in \cite{hor-ser1}. If $n$ is odd, that block must be a Type I block
$\lambda\Delta_{n}$ (with $|\lambda|=1$). If $n=2m$, it can be either a Type I
block or a Type II block $H_{2m}(\mu)$ (with $|\mu|>1$). Of course, the latter
case occurs if and only if the *cosquare of $A$ is similar to $J_{m}%
(\mu)\oplus J_{m}(\bar{\mu}^{-1})$ and $|\mu|>1$.

Suppose the *cosquare of $A$
is similar to
$J_{n}(\lambda^{2})$ with
$\left\vert
\lambda\right\vert =1$. Then
the *congruence canonical
form of $A$ is
$\varepsilon\lambda\Delta_{n}$
and $\varepsilon\in\{1,-1\}$
can be determined as follows.
Let $\hat {A}=\lambda^{-1}A$,
let $S$ be such that
$\Delta_{n}^{-1}\Delta_{n}%
=S^{-1}(\hat{A}^{-\ast}\hat{A})S$, and let $M:=S^{\ast}\hat{A}
S\Delta_{n}^{-1}$. Lemma \ref{Lemma*congrence}(a) tells us that $M$
is similar to a real matrix. Because $\hat{A}$ is indecomposable
under *congruence, in Lemma \ref{Lemma*congrence}(b) either $k=n$ or
$k=0$:
\begin{equation}\label{mdj}
\varepsilon=
  \begin{cases}
    -1 & \text{if all the
eigenvalues of $M$ are
negative} \\
    1 & \text{if no eigenvalue of $M$ is
negative}.
  \end{cases}
\end{equation}

Alternatively, we can employ the algorithm described in Section
\ref{Alternative}. Examine $B:=\Delta_{n}^{-1}S^{\ast}\hat{A}S$, which must be
nonsingular, upper Toeplitz, and real, so
\[
B=%
\begin{bmatrix}
b_{11} & b_{11}^{(2)} & \dots & b_{11}^{(n)}\\
& b_{11} & \ddots & \vdots\\
&  & \ddots & b_{11}^{(2)}\\
0 &  &  & b_{11}%
\end{bmatrix}
\text{,\quad}b_{11}\neq0\text{.}%
\]
The reduction described in Lemma \ref{lemHermitian} is trivial in this case,
and it tells us that $\varepsilon$ is the sign of $b_{11}$.

\begin{example}\label{ex1}
\rm
Consider%
\[
A=\left[
\begin{array}
[c]{cc}%
1 & 2\\
0 & 1
\end{array}
\right]  \text{,}%
\]
whose *cosquare
\[
A^{-\ast}A=\left[
\begin{array}
[c]{cc}%
1 & 2\\
-2 & -3
\end{array}
\right]
\]
is not diagonal, has $-1$ as a double eigenvalue, and hence is similar to
$J_{2}(-1)$. Thus, $A$ is *congruent to $\varepsilon i\Delta_{2}$ with
$\varepsilon=\pm1$.  Let $\hat{A}=-iA$ and verify that%
\[
\left[
\begin{array}
[c]{cc}%
1 & 2i\\
0 & 1
\end{array}
\right]  =\Delta_{2}^{-\ast}\Delta_{2}=S^{-1}(\hat{A}^{-\ast}\hat{A})S\text{
\quad for }S=\left[
\begin{array}
[c]{cc}%
-1 & 0\\
1 & i
\end{array}
\right]  \text{.}%
\]
Then both eigenvalues of $M=S^{\ast}\hat{A}S\Delta_{2}^{-1}=-I_{2}$
are negative, so \eqref{mdj} ensures that $\varepsilon=-1$ and $A$
is
*congruent to $-i\Delta_{2}$. Alternatively,
$B=\Delta_{n}^{-1}S^{\ast}\hat{A}S=-I_{2}$, so the sign of $b_{11}$
is negative and $\varepsilon=-1$.
\end{example}

\begin{example}\label{ex2}
\rm Consider%
\begin{equation}
A=\left[
\begin{array}
[c]{cc}%
0 & 1\\
-1 & 1
\end{array}
\right],  \label{CP-1}%
\end{equation}
whose *cosquare%
\[
A^{-\ast}A=\left[
\begin{array}
[c]{cc}%
-1 & 2\\
0 & -1
\end{array}
\right]
\]
is similar to $J_{2}(-1)$. Thus, $A$ is *congruent to $\varepsilon i\Delta
_{2}$. Let  $\hat{A}=-iA$ and verify that%
\[
\left[
\begin{array}
[c]{cc}%
1 & 2i\\
0 & 1
\end{array}
\right]  =\Delta_{2}^{-\ast}\Delta_{2}=S^{-1}(\hat{A}^{-\ast}\hat{A})S\text{
\quad for }S=\left[
\begin{array}
[c]{cc}%
1 & 0\\
0 & -i
\end{array}
\right]  \text{.}%
\]
Then both eigenvalues of $M=S^{\ast}\hat{A}S\Delta_{2}^{-1}=-I_{2}$
are negative, so \eqref{mdj} ensures that $\varepsilon=-1$ and $A$
is
*congruent to $-i\Delta_{2}$.
\end{example}

\begin{example}\label{ex3}
\rm Suppose $\left\vert \lambda\right\vert =1$ but
$\lambda^{2}\neq-1$. Let $a$ denote the real part of $\lambda$
and consider%
\begin{equation*}
A=\left[
\begin{array}
[c]{cc}%
0 & {\lambda}/a\\
{\lambda}/a & i
\end{array}
\right]=\frac{\lambda}a\left[
\begin{array}
[c]{cc}%
0 & 1\\
1 & a\bar{\lambda}i
\end{array}
\right],
\end{equation*}
whose *cosquare%
\[
A^{-\ast}A= \lambda^{2} \left[
\begin{array}
[c]{cc}%
a{\lambda}i & 1\\
1 & 0
\end{array}
\right]\left[
\begin{array}
[c]{cc}%
0 & 1\\
1 & a\bar{\lambda}i
\end{array}
\right]= \lambda^{2}\left[
\begin{array}
[c]{cc}%
1 & 2a^2i\\
0 & 1
\end{array}
\right]
\]
is similar to
$J_{2}(\lambda^{2})$. Thus,
$A$ is *congruent to
$\varepsilon
\lambda\Delta_{2}$. Let
$\hat{A}=\lambda^{-1}A$
and verify that%
\[
\left[
\begin{array}
[c]{cc}%
1 & 2i\\
0 & 1
\end{array}
\right]
=\Delta_{2}^{-\ast}\Delta_{2}=S^{-1}(\hat{A}^{-\ast}\hat{A})S\text{
\quad for }S:=\left[
\begin{array}
[c]{cc}%
1 & 0\\
0 & {1}/a^{2}
\end{array}
\right].
\]
Then both eigenvalues of
\begin{align*}
M
=S^{\ast}\hat{A}S\Delta_{2}^{-1}
=\left[
\begin{array}
[c]{cc}%
0 & 1/a^{3}\\
1/a^{3} &
{\bar{\lambda}i}/a^{4}
\end{array}
\right]  \left[
\begin{array}
[c]{cc}%
-i & 1\\
1 & 0
\end{array}
\right]   =\left[
\begin{array}
[c]{cc}%
{1}/a^{3}& 0\\
\star & {1}/a^{3}
\end{array}
\right]
\end{align*}
have the same sign as $a$. Thus, \eqref{mdj} ensures that $A$ is
*congruent to
$\lambda\Delta_{2}$ if
$\operatorname{Re}\lambda>0$
and to $-\lambda\Delta _{2}$
if
$\operatorname{Re}\lambda<0$.
\end{example}

\section{Canonical pairs}

\label{sl6}

We now explain how to use the canonical matrices in Theorem \ref{t2}
to obtain the canonical pairs described in Theorem \ref{ths}.

\begin{proof}[Proof of Theorem
\ref{ths}(a)] Each square
matrix $A$ can be expressed
uniquely as the sum of a
symmetric and a
skew-symmetric matrix:
\begin{equation}
A=\mathcal{S}(A)+\mathcal{C}(A),\qquad\mathcal{S}(A):=\tfrac{1}{2}\left(
A+A^{T}\right)  ,\quad\mathcal{C}(A):=\tfrac{1}{2}\left(  A-A^{T}\right)
\text{.} \label{pair16}%
\end{equation}
Since $\mathcal{S}(R^{T}AR)=R^{T}\mathcal{S}(A)R$ and $\mathcal{C}%
(R^{T}AR)=R^{T}\mathcal{C}(A)R$, any congruence that reduces $A$ to a direct
sum%
\[
R^{T}AR=B_{1}\oplus\cdots\oplus B_{k}%
\]
gives a simultaneous congruence of $\left(  \mathcal{S}(A),\mathcal{C}%
(A)\right)  $ that reduces it to a direct sum of pairs
\[
\left(  \mathcal{S}(B_{1}),\mathcal{C}(B_{1})\right)  \oplus\cdots
\oplus\left(  \mathcal{S}(B_{k}),\mathcal{C}(B_{k})\right)  \text{.}%
\]
Theorem \ref{t2}(a) ensures that
$A$ is congruent to a direct
sum of blocks of
the three types
\begin{equation}
2J_{n}(0),\quad
\Gamma_{n},\quad
2H_{2n}(\mu),\label{pair33}%
\end{equation}
in which $0\neq\mu\neq(-1)^{n+1}$ and $\mu$ is determined up to
replacement by $\mu^{-1}$, and that such a decomposition is unique
up to permutation of the direct summands.

Computing the symmetric and skew-symmetric parts of the blocks
(\ref{pair33}) produces the indicated Type 0, Type I, and Type II
canonical pairs in \eqref{table3}.

It remains to prove that the two alternative pairs in \eqref{table4}
may be used instead of the Type II pair
\begin{equation}\label{kne}
\left(  \lbrack
J_{n}(\mu+1)\diagdown
J_{n}(\mu+1)^{T}],\:[J_{n}(\mu
-1)\diagdown-J_{n}(\mu-1)^{T}]\right).
\end{equation}

First suppose that $\mu=-1$, so $n$ is odd (since
$\mu\neq(-1)^{n+1}$) and we have the Type II pair
\begin{equation*}
\left(  \lbrack J_{n}(0)\,
\diagdown\,J_{n}(0)^{T}],\:
[J_{n}(-2)\,\diagdown
\,-J_{n}(-2)^{T}]\right)  \text{.} \label{pair30}%
\end{equation*}
A simultaneous congruence
of this pair via%
\[
\left[
\begin{array}
[c]{cc}%
I_{n} & 0\\
0 & J_{n}(-2)^{-T}%
\end{array}
\right]
\]
transforms it to the pair%
\begin{equation}
\left(  \left[
\begin{array}
[c]{cc}%
0 & J_{n}(0)^{T}J_{n}(-2)^{-T}\\
J_{n}(-2)^{-1}J_{n}(0) & 0
\end{array}
\right]  ,\ \ \left[
\begin{array}
[c]{cc}%
0 & -I_{n}\\
I_{n} & 0
\end{array}
\right]  \right)  \text{.} \label{pair31}%
\end{equation}
Since
$J_{n}(-2)^{-1}J_{n}(0)$ is
similar to $J_{n}(0)$, there
is a nonsingular matrix $S$
such that
$S^{-1}J_{n}(0)S=J_{n}
(-2)^{-1}J_{n}(0)$. Then a simultaneous congruence of (\ref{pair31}) via%
\[
\left[
\begin{array}
[c]{cc}%
S^{-1} & 0\\
0 & S^{T}%
\end{array}
\right]
\]
transforms it to the second of the two alternative pairs in
\eqref{table4}.

Now suppose that $\mu\neq-1$. Let $S$ be a nonsingular matrix such
that
\begin{equation}
S^{-1}J_{n}(\mu-1)J_{n}(\mu+1)^{-1}S=J_{n}\left(  \nu\right)  ,\quad\nu
:=\frac{\mu-1}{\mu+1}\,\text{.} \label{pair21}%
\end{equation}
A simultaneous congruence of
the pair \eqref{kne}
via%
\[%
\begin{bmatrix}
J_{n}(\mu+1)^{-1}S & 0\\
0 & S^{-T}%
\end{bmatrix}
\]
transforms it to the first alternative pair in \eqref{table4}. The
definition (\ref{pair21}) ensures that $\nu\neq1$; $\nu\neq-1$ since
$\mu\neq0$; and $\nu\neq0$ if $n$ is odd since $\mu\neq(-1)^{n+1}$.
Because $\mu$ is determined up to replacement by $\mu^{-1}$, $\nu$
is determined up to replacement by
\[
\frac{\mu^{-1}-1}{\mu^{-1}+1}=\frac{1-\mu}{1+\mu}=-\nu.
\]
\end{proof}

\begin{proof}[Proof of Theorem
\ref{ths}(b)] Each square
complex matrix has a
\textit{Cartesian
decomposition}
\begin{equation}
A=\mathcal{H}(A)+i\mathcal{K}(A),\qquad\mathcal{H}(A):=\tfrac{1}{2}\left(
A+A^{\ast}\right)  ,\quad\mathcal{K}(A):=\tfrac{i}{2}\left(  -A+A^{\ast
}\right)  \label{pair6}%
\end{equation}
in which both $\mathcal{H}(A)$ and $\mathcal{K}(A)$ are Hermitian. Moreover,
if $\mathcal{H}^{\prime}$ and $\mathcal{K}^{\prime}$ are Hermitian matrices
such that $A=\mathcal{H}^{\prime}+i\mathcal{K}^{\prime}$, then $\mathcal{H}%
^{\prime}=\mathcal{H}(A)$ and $\mathcal{K}^{\prime}=\mathcal{K}(A)$. Since
$\mathcal{H}(R^{\ast}AR)=R^{\ast}\mathcal{H}(A)R$ and $\mathcal{K}(R^{\ast
}AR)=R^{\ast}\mathcal{K}(A)R$, any *congruence that reduces $A$ to a direct
sum%
\[
R^{\ast}AR=B_{1}\oplus\cdots\oplus B_{k}%
\]
gives a simultaneous *congruence of $\left(  \mathcal{H}(A),\mathcal{K}%
(A)\right)  $ that reduces it to a direct sum of pairs
\[
\left(  \mathcal{H}(B_{1}),\mathcal{K}(B_{1})\right)  \oplus\cdots
\oplus\left(  \mathcal{H}(B_{k}),\mathcal{K}(B_{k})\right)  \text{.}%
\]
Theorem \ref{t2}(b) ensures that $A$ is congruent to a direct sum of blocks of
the three types
\begin{equation}
2J_{n}(0),\
\lambda\Delta_{n},\text{ and
}H_{2n}(\mu),\text{ in which
}\left\vert
\lambda\right\vert =1\text{ and }\left\vert \mu\right\vert >1 \label{pair35}%
\end{equation}
and that such a decomposition
is unique up to permutation
of the direct summands. The
Cartesian decomposition of
$2J_{n}(0)$ produces the Type
0 pair in \eqref{table5}.

Consider the Type I block $\lambda\Delta_{n}$ with $|\lambda|=1$. If a matrix
$F_{n}$ is nonsingular and $F_{n}^{-\ast}F_{n}$ is similar to $J_{n}%
(\lambda^{2})$, then
$\lambda\Delta_{n}$ is
*congruent to $\pm F_{n}$.
Suppose $\lambda^{2}\neq-1$.
Then
\[
c:=i\frac{1-\lambda^{2}}
{1+\lambda^{2}}=i\frac{\bar{\lambda}(1-\lambda^{2})}
{\bar{\lambda}(1+\lambda^{2})}=
i\frac{\bar{\lambda}-\lambda}{\bar{\lambda}+\lambda}=\frac{\operatorname{Im}
\lambda}{\operatorname{Re} \lambda}
\]
is real and
\[
\lambda^{2}=\frac{{1+ic}}{{1-ic}}\text{.}%
\]
Consider the symmetric matrix
\begin{equation}\label{kts}
F_{n}:=\Delta_{n}(1+ic,i)
=P_{n}\left(
I_{n}+iJ_{n}(c)\right) .
\end{equation}
Then
\[
F_{n}^{-\ast}=\overline{F_{n}^{-1}}=\left(  I_{n}-iJ_{n}(c)\right)  ^{-1}P_{n}%
\]
and
\begin{align*}
F_{n}^{-\ast}F_{n} &  =\left(  I_{n}-iJ_{n}(c)\right)  ^{-1}P_{n}P_{n}\left(
I_{n}+iJ_{n}(c)\right)  \\
&  =\left(  \left(  1-ic\right)  I_{n}-iJ_{n}(0)\right)  ^{-1}\left(  \left(
1+ic\right)  I_{n}+iJ_{n}(0)\right)  \\
&  =\lambda^{2}\left(  I_{n}+a_{1}J_{n}(0)+a_{2}J_{n}(0)^{2}+a_{3}J_{n}%
(0)^{3}+\cdots\right)  \text{,}%
\end{align*}
in which $a_{1}=2i\left(
1+c^{2}\right)  ^{-1}\neq0$.
Thus, $F_{n}^{-\ast }F_{n}$
is similar to
$J_{n}(\lambda^{2})$, so
$\lambda\Delta_{n}$ is
*congruent to
$\pm\Delta_{n}(1+ic,i)$. The
Cartesian decomposition of
$\pm\Delta_{n}(1+ic,i)$
produces the first Type I
pair in \eqref{table5}.

Now suppose that $\lambda^{2}=-1$ and consider the real matrix
\begin{equation}
\label{1aaa}G_{n}:=%
\begin{cases}%
\begin{bmatrix}
0 &  &  &  &  & 1\\
&  &  &  &
\ddd & 1\\
&  &  & 1 & \ddd & \\
&  & -1 & 1 &  & \\
& \ddd &
\ddd &  &  & \\
-1 & 1 &  &  &  & 0
\end{bmatrix}
\!\!\!%
\begin{matrix}
\left.  \rule{0pt}{23pt}\right\}  m\\[17pt]%
\left.  \rule{0pt}{23pt}\right\}  m
\end{matrix}
& \text{if $n=2m$},\\[50pt]
\begin{bmatrix}
0 &  &  &  &  &  & 1\\
&  &  &  &  & \ddd & 1\\
&  &  &  & 1 & \ddd & \\
&  &  & \fbox{1} & 1 &  & \\
&  & 1 & 0 &  &  & \\
& \ddd &
\ddd &  &  &  & \\
1 & 0 &  &  &  &  & 0
\end{bmatrix}
\!\!\!%
\begin{matrix}
\left.  \rule{0pt}{23pt}\right\}  m\\[17pt]%
\left.  \!\!\!\right\}
1\\[5pt]
\left.  \rule{0pt}{23pt}\right\}  m
\end{matrix}
& \text{if $n=2m+1$}%
\end{cases}
\end{equation}
(the boxed unit is at the
center). Since
$G_{n}^{-T}G_{n}$ is similar
to $J_{n}(-1)$,
$\lambda\Delta_{n}$ is
*congruent to $\pm G_{n}$.
The Cartesian decomposition
of $\pm G_{n}$ produces the
second Type I pair in
\eqref{table5}.

The Type II block $H_{2n}(\mu)$ with $\left\vert \mu\right\vert >1$ is
*congruent to $H_{2n}(\bar{\mu}^{-1})$ because their *cosquares are both
similar to $J_{n}(\mu)\oplus J_{n}(\bar{\mu}^{-1})$. Represent $\bar{\mu}%
^{-1}$, a point in the open unit disk with the origin is omitted, as%
\[
\bar{\mu}^{-1}=\frac{1+i\nu}{1-i\nu},
\]
in which $\nu$ is in the open upper half plane with the point $i$ omitted. In
fact,%
\begin{equation}
\nu=\frac{2\operatorname{Im}\mu+i\left(  \left\vert \mu\right\vert
^{2}-1\right)  }{\left\vert \mu+1\right\vert ^{2}}:=a+ib\neq i,\quad b>0,\quad
a,b\in\mathbb{R}\text{.}\label{pair40}%
\end{equation}
We have%
\begin{align*}
\left(  I_{n}+iJ_{n}(\nu)\right)  \left(  I_{n}-iJ_{n}(\nu)\right)  ^{-1} &
=\left(  \left(  1+i\nu\right)  I_{n}+iJ_{n}(0)\right)  \left(  \left(
1-i\nu\right)  I_{n}-iJ_{n}(0)\right)  ^{-1}\\
&  =\bar{\mu}^{-1}\left(  I_{n}+a_{1}J_{n}(0)+a_{2}J_{n}(0)^{2}+a_{3}%
J_{n}(0)^{3}+\cdots\right),
\end{align*}
in which $a_{1}=2i\left(  1+\nu^{2}\right)  ^{-1}\neq0$. Thus,
\[
\left(  I_{n}+iJ_{n}(\nu)\right)  \left(  I_{n}-iJ_{n}(\nu)\right)  ^{-1}%
\]
is similar to $J_{n}(\bar{\mu}^{-1})$. Let $S$ be a nonsingular matrix such
that
\[
S^{-1}J_{n}(\bar{\mu}^{-1})S=\left(  I_{n}+iJ_{n}(\nu)\right)  \left(
I_{n}-iJ_{n}(\nu)\right)  ^{-1}%
\]
and compute the following *congruence of $H_{2n}(\bar{\mu}^{-1})$:
\begin{multline*}%
\begin{bmatrix}
S\left(  I_{n}-iJ_{n}(\nu)\right)   & 0\\
0 & S^{-\ast}%
\end{bmatrix}
^{\ast}%
\begin{bmatrix}
0 & I_{n}\\
J_{n}(\bar{\mu}^{-1}) & 0
\end{bmatrix}%
\begin{bmatrix}
S\left(  I_{n}-iJ_{n}(\nu)\right)   & 0\\
0 & S^{-\ast}%
\end{bmatrix}
\\
=%
\begin{bmatrix}
0 & I_{n}+iJ_{n}(\nu)^{\ast}\\
I_{n}+iJ_{n}(\nu) & 0
\end{bmatrix}
\text{.}%
\end{multline*}
The Cartesian decomposition
of this matrix produces the
Type II pair in
\eqref{table5} with the
parameters $a$ and $b$
defined in (\ref{pair40}).
\end{proof}

\medskip For a given $\lambda$ with $\left\vert \lambda\right\vert =1$, the
$\pm$ signs associated with
the Type I canonical pairs in
\eqref{table5} can be
determined using the
algorithms in either Section
3 or Section 4.

For example, suppose $n=2$ and $\lambda=i$. The matrix%
\[
G_{2}=\left[
\begin{array}
[c]{cc}%
0 & 1\\
-1 & 1
\end{array}
\right]
\]
defined in \eqref{1aaa} was
analyzed in Example
\ref{ex2}. We found that
$i\Delta_{2}$ is *congruent
to $-G_{2}$, so the
*congruence canonical pair
associated with $i\Delta_{2}$
is
\[
-\left(  \Delta_{2}(0,1),\Delta_{2}(1,0)\right)  =-\left(  \left[
\begin{array}
[c]{cc}%
0 & 0\\
0 & 1
\end{array}
\right]  ,\left[
\begin{array}
[c]{cc}%
0 & 1\\
1 & 0
\end{array}
\right]  \right)  \text{.}%
\]

As a second example, suppose $n=2$ and  $\left\vert
\lambda\right\vert =1$, but $\lambda^{2}\neq-1$. Let $\lambda=a+ib$
($a,b\in\mathbb R$).
The matrix%
\[
F_{2}=\Delta_{2}(1+ib/a,i)
=\Delta_{2}({\lambda}/a,i)=\left[
\begin{array}
[c]{cc}%
0 & {\lambda}/a\\
{\lambda}/a & i
\end{array}
\right]
\]
defined in \eqref{kts}, was analyzed in Example \ref{ex3}. We found
that $\lambda\Delta_{2}$ is
*congruent to $F_{2}$ if
$a>0$, and to $-F_{2}$ if
$a<0$. Thus, the *congruence
canonical pair associated
with $\lambda\Delta_{2}$ is%
\[
\left( \Delta_{2}(1,0),\:
\Delta_{2}\left( b/a,1\right)
\right)  =\left(  \left[
\begin{array}
[c]{cc}%
0 & 1\\
1 & 0
\end{array}
\right],\;\left[
\begin{array}
[c]{cc}%
0 & b/a\\
b/a & 1
\end{array}
\right]  \right)
\]
if $\operatorname{Re}\lambda>0$, and is%
\[
-\left(  \left[
\begin{array}
[c]{cc}%
0 & 1\\
1 & 0
\end{array}
\right]  ,\left[
\begin{array}
[c]{cc}%
0 & b/a\\
b/a & 1
\end{array}
\right]  \right)
\]
if
$\operatorname{Re}\lambda<0$.


\begin{thebibliography}{99}                                                                                               %
\bibitem {bal}C. S. Ballantine, Cosquares: complex and otherwise, \emph{Linear
and Multilinear Algebra} 6 (1978) 201--217.

\bibitem {cor}B. Corbas and G. D. Williams, Bilinear forms over an
algebraically closed field, \textit{J. Pure Appl. Algebra} 165 (3) (2001) 225--266.

\bibitem {gab}P. Gabriel, Appendix: degenerate bilinear forms, \emph{J.
Algebra} 31 (1974) 67--72.

\bibitem {gan}F. R. Gantmacher, \textit{The Theory of Matrices}, vol. I ,
Chelsea, New York, 2000.

\bibitem {HJ1}R. A. Horn and C. R. Johnson, \emph{Matrix Analysis}, Cambridge
University Press, New York, 1985.

\bibitem {horn}R. A. Horn and C. R. Johnson, \textit{Topics in Matrix
Analysis}, Cambridge University Press, New York, 1991.

\bibitem {HP}R. A. Horn and G. Piepmeyer, Two applications of the theory of
primary matrix functions, \emph{Linear Algebra Appl.} 361 (2003) 99-106.

\bibitem {hor-ser}R. A. Horn and V. V. Sergeichuk, Congruence of a square
matrix and its transpose, \textit{Linear Algebra Appl.} 389 (2004) 347--353.

\bibitem {hor-ser1}R. A. Horn and V. V. Sergeichuk, A regularizing algorithm
for matrices of bilinear and sesquilinear forms, \emph{Linear Algebra Appl.} 412 (2006) 380--395.

\bibitem {Hua}L. K. Hua, On the theory of automorphic functions of a matrix
variable II---The classification of hypercircles under the symplectic group,
\emph{Amer. J. Math.} 66 (1944) 531-563.

\bibitem {Ikramov}Kh. D. Ikramov, On the inertia law for normal matrices,
\emph{Doklady Math.} 64 (2001) 141--142.

\bibitem {JF}C. R. Johnson and S. Furtado, A generalization of Sylvester's law
of inertia, \emph{Linear Algebra Appl.} 338 (2001) 287--290.

\bibitem {lan-rod}P. Lancaster, L. Rodman, Canonical forms for
symmetric/skew-symmetric real matrix pairs under strict equivalence and
congruence, \emph{Linear Algebra Appl.} 406 (2005) 1--76.

\bibitem {L-R}P. Lancaster and L. Rodman, Canonical forms for Hermitian matrix
pairs under strict equivalence and congruence, \emph{SIAM Review} 47 (2005) 407-443.

\bibitem {leewei}J. M. Lee and D. A. Weinberg, A note on canonical forms for
matrix congruence, \textit{Linear Algebra Appl.} 249 (1996) 207--215.

\bibitem {rie}C. Riehm, The equivalence of bilinear forms, \emph{J. Algebra}
31 (1974) 45--66.

\bibitem {RSF}C. Riehm and M. Shrader--Frechette, The equivalence of
sesquilinear forms, \emph{J. Algebra} 42 (1976) 495--530.

\bibitem {robinson}D. W. Robinson, An alternative approach to unitoidness, \emph{Linear Algebra
Appl.} 413 (2006) 72-80.

\bibitem {roi}A. V. Roiter, Bocses with involution, in:
\textit{Representations and Quadratic Forms} (Ju. A. Mitropol$^{\prime}$skii,
Ed.), Inst. Mat. Akad. Nauk Ukrain. SSR, Kiev, 1979, 124--128 (in Russian).

\bibitem {ser1}V. V. Sergeichuk, Classification problems for system of forms
and linear mappings, \emph{Math. USSR, Izvestiya}\/ 31 (3) (1988) 481--501.

\bibitem {thom}R. C. Thompson, Pencils of complex and real symmetric and skew
matrices, \textit{Linear Algebra Appl.} 147 (1991) 323--371.

\bibitem {Wall}G. E. Wall, On the conjugacy classes in the unitary, symplectic
and orthogonal groups, \emph{J. Aust. Math. Soc.} 3 (1963) 1-62.

\bibitem {wat}W. C. Waterhouse, The number of congruence classes in
$M_{n}(F_{q})$, \textit{Finite Fields Appl.} 1 (1995) 57--63.
\end{thebibliography}
\end{document}